\pdfoutput=1
\RequirePackage{ifpdf}
\ifpdf 
\documentclass[pdftex]{sigma}
\else
\documentclass{sigma}
\fi

\usepackage{bbm}

\input xy
\xyoption{all}

\newcommand{\Dleft}{[\hspace{-1.5pt}[}
\newcommand{\Dright}{]\hspace{-1.5pt}]}
\newcommand{\SN}[1]{\Dleft #1 \Dright}
\newcommand{\Id}{\mathbbmss{1}}
\newcommand{\rmd}{\textnormal{d}}
\newcommand{\rmh}{\textnormal{h}}
\newcommand{\rml}{\textnormal{l}}
\newcommand{\rmA}{\textnormal{A}}
\newcommand{\rmg}{\textnormal{g}}
\newcommand{\Lie}{\textnormal{Lie}}

\newcommand{\catname}[1]{\textnormal{\texttt{#1}}}

\def\Sec{\mathtt{Sec}}
\def\Vect{\mathtt{Vect}}
\def\sT{{\mathtt T}}
\def\xd{\operatorname{d}}

\numberwithin{equation}{section}

\newtheorem{Theorem}{Theorem}[section]
\newtheorem{Proposition}[Theorem]{Proposition}
 { \theoremstyle{definition}
\newtheorem{Definition}[Theorem]{Definition}
\newtheorem{Example}[Theorem]{Example}
\newtheorem{Remark}[Theorem]{Remark} }

\begin{document}

\allowdisplaybreaks

\renewcommand{\thefootnote}{$\star$}

\newcommand{\arXivNumber}{1502.06092}

\renewcommand{\PaperNumber}{090}

\FirstPageHeading

\ShortArticleName{Graded Bundles in the Category of Lie Groupoids}

\ArticleName{Graded Bundles in the Category of Lie Groupoids\footnote{This paper is a~contribution to the Special Issue
on Poisson Geometry in Mathematics and Physics.
The full collection is available at \href{http://www.emis.de/journals/SIGMA/Poisson2014.html}{http://www.emis.de/journals/SIGMA/Poisson2014.html}}}

\Author{Andrew James {BRUCE}~$^\dag$,  Katarzyna {GRABOWSKA}~$^\ddag$ and Janusz {GRABOWSKI}~$^\dag$}

\AuthorNameForHeading{A.J.~Bruce, K.~Grabowska and J.~Grabowski}

\Address{$^\dag$~Institute of Mathematics, Polish Academy of Sciences, Poland}
\EmailD{\href{mailto:andrewjamesbruce@googlemail.com}{andrewjamesbruce@googlemail.com}, \href{mailto:jagrab@impan.pl}{jagrab@impan.pl}}

\Address{$^\ddag$~Faculty of Physics, University of Warsaw, Poland}
\EmailD{\href{mailto:konieczn@fuw.edu.pl}{konieczn@fuw.edu.pl}}

\ArticleDates{Received February 25, 2015, in f\/inal form November 05, 2015; Published online November 11, 2015}

\Abstract{We def\/ine and make initial study of  Lie groupoids equipped with a compatible homogeneity (or graded bundle) structure, such objects we will refer to as \emph{weighted Lie groupoids}.  One can think of weighted Lie groupoids as graded  manifolds in the category of Lie groupoids.  This is a very rich geometrical theory with numerous natural examples. Note that $\mathcal{VB}$-groupoids, extensively studied in the recent literature, form just the particular case of weighted Lie groupoids of degree one. We examine the Lie theory related to weighted groupoids and \emph{weighted Lie algebroids}, objects def\/ined in a previous publication of the authors, which are graded  manifolds in the category of Lie algebroids, showing that they are naturally related via dif\/ferentiation and integration.  In this work we also make an initial study of \emph{weighted Poisson--Lie groupoids} and \emph{weighted Lie bi-algebroids}, as well as \emph{weighted Courant algebroids}.}


\Classification{22A22; 55R10; 	58E40;  58H05}

\renewcommand{\thefootnote}{\arabic{footnote}}
\setcounter{footnote}{0}

\section{Introduction}\label{sec:Intro}
Lie groupoids and Lie algebroids are ubiquitous throughout dif\/ferential geometry and are playing an ever increasing r\^{o}le in mathematical physics.  Lie groupoids provide a unifying framework to discuss diverse topics in modern geometry including the theory of group actions, foliations, Poisson geometry, orbifolds, principal bundles, connection theory  and so on.  The inf\/initesimal counterpart to Lie groupoids are  Lie algebroids. The Lie theory here is very rich and not as simple as the corresponding Lie theory for   Lie groups and Lie algebras. In particular, while it is true  that  every Lie groupoid  can be dif\/ferentiated to obtain a Lie algebroid, the reverse procedure of  global integration has an obstruction \cite{Crainic:2003,Mackenzie:1987}. Not all Lie algebroids can be globally  integrated to obtain a Lie groupoid,  although one can always integrate Lie algebroids to `local' Lie groupoids.

Another important notion in modern geometry is that of a graded manifold. Such notions  have their conception in the `super' context via the BV--BRST formalism of gauge theories in physics. We remark that Lie algebroids appear in physics as symmetries of f\/ield theories that do not arise from Lie groups or Lie algebras. Moreover, these symmetries cannot always be directly separated from the space of f\/ields. Such symmetries are naturally accommodated in the  BV--BRST formalism. The concept of  a graded manifold has  been put to good use in describing Lie algebroids, Lie bi-algebroids, Courant algebroids and related notions~\cite{Roytenberg:2001,Roytenberg:2006,Voronov:2001qf,Voronov:2012}.    In the purely commutative setting, Grabowski and Rotkiewicz~\cite{Grabowski2012} def\/ine what they referred to as \emph{graded bundles}. Loosely, a graded bundle is a natural generalisation of the concept of a vector bundle. We will discuss graded bundles in a little more detail shortly.

There has been some recent interest in $\mathcal{VB}$-groupoids and $\mathcal{VB}$-algebroids, see for example \cite{Brahic:2014,Bursztyn:2014,Gracia-Saz:2009,Gracia-Saz:2010}.  From a categorical point of view $\mathcal{VB}$-groupoids are vector bundles in the category of Lie groupoids and similarly $\mathcal{VB}$-algebroids are vector bundles in the category of Lie algebroids.   It is known, via \cite{Brahic:2014,Bursztyn:2014}, that the Lie functor restricts to the category of  $\mathcal{VB}$-groupoids and $\mathcal{VB}$-algebroids; that is we can dif\/ferentiate  $\mathcal{VB}$-groupoids to get $\mathcal{VB}$-algebroids and $\mathcal{VB}$-algebroids integrate to $\mathcal{VB}$-groupoids. Let us mention that $\mathcal{VB}$-groupoids and  $\mathcal{VB}$-algebroids have been studied in relation to Poisson geometry and representations. Indeed, representations of Lie groupoids on vector bundles gives rise to $\mathcal{VB}$-groupoids.

As graded bundles are generalisations of vector bundles, it seems natural that one should study graded bundles in the category of Lie groupoids and  Lie algebroids, including their Lie theory. The f\/irst work in this direction was by the authors of this paper, where various descriptions of so-called \emph{weighted Lie algebroids} were given in~\cite{Bruce:2014}. The motivation for that work was to uncover a practical notion of a ``higher Lie algebroid'' suitable for higher-order geometric mechanics in the spirit of Tulczyjew. We presented details of the Lagrangian and Hamiltonian formalisms on graded bundles using weighted Lie algebroids in a separate publication~\cite{Bruce:2014b}. As an application, the higher-order Euler--Lagrange equations on a Lie algebroid were derived completely geometrically, and in full agreement with the independent   approach of Mart\'{i}nez using variational calculus~\cite{Martinez:2015}. This and other results have convinced us of the potential for further applications of weighted Lie algebroids.

The natural question of integrating weighted Lie algebroids was not posed at all in~\cite{Bruce:2014} or~\cite{Bruce:2014b}. We address this question in this paper. We  def\/ine \emph{weighted Lie groupoids} as Lie groupoids that carry a compatible homogeneity structure, i.e., an action of the monoid ${\mathbb R}$ of multiplicative reals; we will make this precise in due course.   We show that these structures, which are a~generalisation of $\mathcal{VB}$-groupoids, are the objects that integrate weighted Lie algebroids.   Just as the tangent bundle of a Lie groupoid is a canonical example of a~$\mathcal{VB}$-groupoid,  the $k$-th order  higher tangent bundle of a~Lie groupoid is a canonical example of a weighted Lie groupoid.

The results found in this paper can  be viewed as the `higher order' generalisations of the results found in~\cite{Bursztyn:2014}. In particular, Bursztyn et al.~\cite{Bursztyn:2014} treat $\mathcal{VB}$-groupoids/algebroids as Lie groupoids/algebroids equipped with a compatible regular homogeneity structure.  The idea of simplifying $\mathcal{VB}$-`objects' by describing then in terms of regular homogeneity structures goes back to~\cite{Grabowski:2006}, where symplectic structures were studied  in this context.  By passing to graded bundles, we do not really  further simplify the notion of $\mathcal{VB}$-`objects', but by understanding them as part of a larger class of structures we simplify working with them.

Following our intuition, we also examine the notion of \emph{weighted Poisson--Lie groupoids} as  Poisson--Lie groupoids equipped with a compatible homogeneity structure. The inf\/initesimal versions of weighted Poisson--Lie groupoids are \emph{weighted Lie bi-algebroids}, a notion we carefully def\/ine in this paper. Associated with any Lie bi-algebroid is a Courant algebroid. We show that this notion very naturally passes over to the weighted case and motivates a more general notion of a \emph{weighted Courant algebroid}, which is a natural generalisation of a~$\mathcal{VB}$-Courant algebroid.  The use of a~compatible homogeneity structure provokes the question of replacing the monoid~${\mathbb R}$ by its subgroup~${\mathbb R}^\times$ of multiplicative reals $\ne 0$. We claim that this is a natural way of def\/ining contact and Jacobi groupoids, see~\cite{Bruce:2015}.

Summing up, one can say that in this paper we generalise the notion of $\mathcal{VB}$-`objects'  by passing to the category of graded bundles and categorical objects therein.  On the other hand, these concepts and their applications are far from being fully exploited. This work is only the f\/irst step in this direction.

\looseness=-1
\textbf{Our use of supermanifolds.} We recognise the power and elegance of the framework of supermanifolds in the context of algebroid-like objects and we will exploit this formalism. In particular we will make use of the so called $Q$-manifolds, that is supermanifolds equipped with a homological vector f\/ield. Although we will not dwell on fundamental issues from the theory of supermanifolds, we will technically follow the ``Russian school''  and understand supermanifolds in terms of locally superringed spaces. However, for the most part the intuitive and correct understanding of a supermanifold as a `manifold' with both commuting and anticommuting coordinates will suf\/f\/ice. When necessary we will denote the Grassmann parity of an object by `tilde'.

\textbf{Arrangement of paper.} In Section~\ref{sec:Preliminaries} we brief\/ly recall the necessary parts of theory of graded bundles, weighted Lie algebroids and Lie groupoids as needed in the rest of this paper.  In Section~\ref{sec:weighted Lie groupoids} we introduce the main objects of study, that is weighted Lie groupoids. Weighted Lie theory is the subject of Section~\ref{sec:Lie theory}. Finally in Section~\ref{sec:Weighted Poisson Groupoids} we apply some of the ideas developed earlier in this paper to weighted Poisson--Lie groupoids, weighted Lie bi-algebroids and weighted Courant algebroids.  We end  with some remarks on contact and Jacobi groupoids.

\section{Preliminaries}\label{sec:Preliminaries}

\looseness=-1
In this section we brief\/ly recall parts of the theory of graded bundles, $n$-tuple graded bundles and weighted algebroids as needed in later sections of this paper. Everything in this paper will be in the smooth category.  The interested reader should consult the original literature \cite{Bruce:2014,Grabowski2013,Grabowski:2006,Grabowski2012} for details, such as proofs of the statements made in this section.  We will also set some notation regarding Lie groupoids and  recall  the groupoid/algebroid version of Lie's second theorem, which was f\/irst proved by Mackenzie and Xu~\cite{Mackenzie:2000}. For an overview of the general theory of Lie groupoids and Lie algebroids the reader can consult Mackenzie~\cite{Mackenzie2005}. We will also very brief\/ly recall the notion of a $Q$-manifold and a $QS$-manifold as needed throughout this paper.

\subsection[Graded and $n$-tuple graded bundles]{Graded and $\boldsymbol{n}$-tuple graded bundles}

Manifolds and supermanifolds that carry various extra gradings on their structure sheaf are now an established part of modern geometry and mathematical physics. The general theory of graded manifolds in our understanding was formulated  by Voronov  in~\cite{Voronov:2001qf};  we must also mention the  works of Kontsevich \cite{Kontsevich:2003}, Mehta~\cite{Mehta:2006}, Roytenberg~\cite{Roytenberg:2001} and \v{S}evera~\cite{Severa:2005} where various notions of a~$\mathbb{Z}$-graded supermanifold appear. The graded structure on such (super)manifolds is conveniently encoded in the \emph{weight vector field} whose action via the Lie derivative counts the degree of tensor and tensor-like objects on the (super)manifold.

In this introductory section we will concentrate our attention on just genuine manifold and only sketch the theory for supermanifolds, as the extension of the results here to supermanifolds is almost straightforward. In later sections we will make use of supermanifolds that carry additional gradings coming from a \emph{homogeneity structure} which will be def\/ined in a moment.

An important class of graded manifolds are those that carry non-negative grading.  It will be convenient to denote homogeneous local coordinates in the form $(y^a_w)$ (or $(y^w_a)$), where $w=0,1,\dots,k$ is the weight of $y^w_a$.
A canonical example of such a structure is the bundle $\sT^kM=J^k_0({\mathbb R},M)$ of $k$-velocities, i.e., $k$th-jets (at $0$) of curves $\gamma\colon {\mathbb R}\to M$. We will furthermore require that (like in the case of $\sT^kM$) this grading is associated with a smooth action $h\colon {\mathbb R}_{\ge 0}\times F\to F$ of the monoid $({\mathbb R}_{\ge 0},\cdot)$ of multiplicative non-negative reals on a manifold $F$. Let us recall in this context that a function $f\in C^\infty(F)$ we call \emph{homogeneous of degree $k\in{\mathbb R}$} if
\begin{gather}\label{hmg}
h_t^*(f):=f(h(t,\cdot))=t^k f
\end{gather}
for all $t>0$. We call $k$  the \emph{weight}  or \emph{degree} of $f$.

Such actions are known as  \emph{homogeneity structures} in the terminology of  Grabowski and Rot\-kie\-wicz~\cite{Grabowski2012} who proved that only \emph{non-negative integer weights} are allowed, so the algebra $\mathcal{A}(F)\subset C^\infty(F)$ spanned by homogeneous function  has a canonical $\mathbb{N}$-grading, $\mathcal{A}(F) = \bigoplus_{i \in \mathbb{N}}\mathcal{A}^{i}(F)$. This algebra  is referred to as the \emph{algebra of polynomial functions} on~$F$.  This action reduced to ${\mathbb R}_{>0}$ is the one-parameter group of dif\/feomorphism integrating the weight vector f\/ield, thus the weight vector f\/ield is in this case \emph{${\rm h}$-complete}~\cite{Grabowski2013}.  Note also that the homogeneity structure always has a canonical extension to the action $h\colon {\mathbb R}\times F\to F$ of the monoid $({\mathbb R},\cdot)$ such that any homogeneous function $f$ of degree $k\in\mathbb{N}$ satisf\/ies~(\ref{hmg}) this time for all $t\in{\mathbb R}$; it will be convenient to speak about a homogeneity structure as  this extended action.

Importantly, we have  that for $t \neq 0$ the action $h_{t}$ is a dif\/feomorphism of $F$ and when $t=0$ it is a smooth surjection $\tau=h_0$ onto a submanifold $F_{0}=M$, with the f\/ibres being dif\/feomorphic to $\mathbb{R}^{N}$ (cf.~\cite{Grabowski2012}).  Thus, the objects obtained are particular kinds of \emph{polynomial bundles} $\tau\colon F\to M$, i.e., f\/ibrations which locally look like $U\times{\mathbb R}^N$ and the change of coordinates (for a certain choice of an atlas) are polynomial in~${\mathbb R}^N$. For this reason graded manifolds with non-negative weights and h-complete weight vector f\/ields are also known as \emph{graded bundles}~\cite{Grabowski2012}. Furthermore, the h-completeness condition  implies that  graded bundles are determined by the algebra of homogeneous functions on them.  Canonical examples of  graded bundles are higher tangent bundles~$\sT^kM$.
The canonical coordinates $(x^a,\dot x^b,\ddot x^c,\dots)$ on $\sT^kM$, associated with local coordinates $(x^a)$ on~$M$, have degrees, respectively, $0,1,2,\dots$, so the homogeneity structure reads
\begin{gather*}
h^{*}_t\big(x^a,\dot x^b,\ddot x^c,\dots\big)=\big(x^a,t\dot x^b,t^2\ddot x^c,\dots\big) .
\end{gather*}

A fundamental result is that any smooth action of the multiplicative monoid  $({\mathbb R},\cdot)$ on a~mani\-fold leads to a $\mathbb{N}$-gradation of that manifold. A little more carefully, the category of graded bundles is  equivalent to the category of $(\mathbb{R}, \cdot)$-manifolds and equivariant maps. A canonical construction of this correspondence goes as follows.
Take $ p \in F$ and consider $t \mapsto h_{t}(p)$ as a~smooth curve $\gamma_h^p$ in $F$. This curve meets~$M$ for $t=0$ and is constant for~$p\in M$.

\begin{Theorem}[\cite{Grabowski2012}]
For any homogeneity structure $h\colon {\mathbb R}\times F\to F$ there is $k\in\mathbb{N}$ such that
the map
\begin{gather}
\label{Phi}
\Phi_h^k\colon \ F\to \sT^k F_{|M} ,\qquad \Phi^k_h(p)={\mathtt j}^k_0(\gamma_h^p) ,
\end{gather}
is an equivariant $($with respect to the monoid actions of $({\mathbb R},\cdot)$ on~$F$ and~$\sT^k F)$ embedding of~$F$ into a graded submanifold of the graded bundle $\sT^k F_{|M}$. In particular, there is an atlas on~$F$ consisting of homogeneous function.
\end{Theorem}

Any $k$ described by the above theorem we call a \emph{degree} of the homogeneity structure~$h$. We stress that  a graded bundle is not just a manifold with consistently def\/ined homogeneous local coordinates. The additional condition is that the weight vector f\/ield encoding the graded structure is  h-complete; that is the associated one-parameter group of dif\/feomorphisms  can be extended to an action of the monoid~$(\mathbb{R}, \cdot)$.

\begin{Remark}
The above theorem has a counterpart for graded supermanifolds, the proof of which we will quickly sketch here. Following Voronov~\cite{Voronov:2001qf}, we do not require that weights induce parity. The compatibility condition just means that $\mathbb{N}$-weights commute with $\mathbb{Z}_2$-grading. In other words, a homogeneity structure on a supermanifold ${\mathcal M}$ is a smooth action $h\colon {\mathbb R}\times{\mathcal M}\to{\mathcal M}$ of the monoid $({\mathbb R},\cdot)$ such that $h_t\colon {\mathcal M}\to{\mathcal M}$ are morphisms, i.e., respect the parity. Equivalently, we have an $({\mathbb R},\cdot)$-action on
the super-algebra ${\mathcal O}_{\mathcal M}=C^\infty({\mathcal M})$ by homomorphisms.

The method of the proof of \cite[Theorem~1.13]{KMS} can be easily adapted to the supermanifold case to show that ${\mathcal M}_0=h_0({\mathcal M})$ is a~submanifold of~${\mathcal M}$. We can also adapt the methods of the proofs of Lemma~4.1 and Theorem~4.1 in~\cite{Grabowski2012}.
First, def\/ine $H(t)\colon \sT{\mathcal M}_{|{\mathcal M}_0}\to\sT{\mathcal M}_{|{\mathcal M}_0}$ as the derivative $\sT {h_t}$ restricted to $\sT{\mathcal M}_{|{\mathcal M}_0}$. We have $H(t)H(s)=H(ts)$, so $Q_0=H(0)$ is a projector. Since both, $Q_0$ and $\Id-Q_0$ are projectors whose rank cannot fall locally, the rank of $Q_0$ is constant along ${\mathcal M}_0$ (assuming ${\mathcal M}$ is connected) and $E_0=Q_0(\sT{\mathcal M}_{|{\mathcal M}_0})$ is a vector subbundle of~$\sT{\mathcal M}_{|{\mathcal M}_0}$. One can show inductively that
$Q_r=\frac{\xd^r H}{\xd t^r}(0)$ is a~projector operator def\/ined on the kernel of $Q_{r-1}$ and we put
$E_r$ to be its range. According to the argument similar to that of~\cite[Proposition~4.2]{Grabowski2012}, $E_{k+1}=\sT{\mathcal M}_0$ for some~$k$. We have therefore
a vector bundle decomposition
\begin{gather*}
\sT{\mathcal M}_{|{\mathcal M}_0}=E_0\oplus_{{\mathcal M}_0}\dots\oplus_{{\mathcal M}_0}E_{k+1}
\end{gather*}
such that, roughly speaking,
\begin{gather*}
\frac{\partial}{\partial x}\frac{\partial^i h}{\partial t^i}(0,x) ,
\end{gather*}
viewed as a liner map on $\sT{\mathcal M}_{|{\mathcal M}_0}$, vanish on $E_r$ for $r>i$ and are identical on $E_i$.
This implies that the analog of the map~(\ref{Phi}), $\Phi^k_h\colon {\mathcal M}\to\sT^k{\mathcal M}_{|{\mathcal M}_0}$.
whose coordinate form is
\begin{gather*}
\Phi^k_h(x)=\left(h(0,x),\frac{\partial h}{\partial t}(0,x),\dots,\frac{\partial^k h}{\partial t^k}(0,x)\right) ,
\end{gather*}
is of maximal rank at ${\mathcal M}_0$. Since it intertwines the actions of $({\mathbb R},\cdot)$ (cf.~\cite[Lemma~4.1]{Grabowski2012}), it is actually an embedding into a graded submanifold of $\sT^k{\mathcal M}_{|{\mathcal M}_0}$.
Indeed, the maximal rank argument is suf\/f\/icient for odd coordinates, while for even ones we can refer to \cite{Grabowski2012}.
\end{Remark}

Morphisms between graded bundles are necessarily polynomial in the non-zero weight coordinates and respect the weight. Such morphisms can be characterised by the fact that they relate the respective weight vector f\/ields,  or equivalently the respective homogeneity structures  \cite{Grabowski2012}.

\begin{Remark}
It is possible to consider manifolds with gradations that do not lead to h-complete weight vector f\/ields. From the point of view of this paper such manifolds are less rigid in their structure and will exhibit pathological behavior.  We remark that the def\/inition given by Voronov~\cite{Voronov:2001qf} of a~non-negatively graded supermanifold states that the Grassmann even coordinates are `cylindrical'.  Similarly, Roytenberg~\cite{Roytenberg:2001} also insists on this `cylindrical' condition in his def\/inition of an $N$-manifold.   This together with the fact that functions of Grassmann odd coordinates are necessarily polynomial, means that the weight vector f\/ields on  non-negatively graded supermanifolds and the closely related $N$-manifolds are h-complete. One should also note that \v{S}evera~\cite{Severa:2005} takes as his \emph{definition} of an $N$-manifold a supermanifold equipped with an action of $(\mathbb{R}, \cdot)$ such that $-1  \in \mathbb{R}$ acts as the parity operator (it f\/lips sign of any Grassmann odd coordinate). Note however, that \v{S}evera does not of\/fer a proof that homogeneous local coordinates can always be found and so it is not immediately clear if his notion of an $N$-manifold exactly coincides with that of  Roytenberg. This equivalence follows from minor modif\/ication of the arguments made in the previous remark.
\end{Remark}

A graded bundle  of degree $k$  admits a sequence of  polynomial f\/ibrations, where a point of~$F_{l}$ is a class of the points of $F$ described  in an af\/f\/ine coordinate system by the coordinates of weight $\leq l$, with the obvious tower of  surjections
\begin{gather*}
F=F_{k} \stackrel{\tau^{k}}{\longrightarrow} F_{k-1} \stackrel{\tau^{k-1}}{\longrightarrow}   \cdots \stackrel{\tau^{3}}{\longrightarrow} F_{2} \stackrel{\tau^{2}}{\longrightarrow}F_{1} \stackrel{\tau^{1}}{\longrightarrow} F_{0} = M,
\end{gather*}
 where the coordinates on $M$ have zero weight. Note that  $F_{1} \rightarrow M$ is a linear f\/ibration and the other f\/ibrations $F_{l} \rightarrow F_{l-1}$ are af\/f\/ine f\/ibrations in the sense that the changes of local coordinates for the f\/ibres are linear plus and additional additive terms of appropriate weight.  The model f\/ibres here are $\mathbb{R}^{N}$ (cf.~\cite{Grabowski2012}).  We will also use on occasion $\tau := \tau^{k}_{0}\colon F_{k} \rightarrow M$.

Canonical examples of graded bundles are, for instance, vector bundles, $n$-tuple vector bundles, higher tangent bundles~$\sT^kM$, and multivector bundles $\wedge^n\sT E$ of vector bundles $\tau\colon E\to M$ with respect to the projection $\wedge^n\sT\tau\colon \wedge^n\sT E\to \wedge^n\sT M$ (see~\cite{Grabowska2014}).  If the weight is constrained to be either zero or one, then the weight vector f\/ield is precisely a  vector bundle structure on~$F$ and will be generally referred  to as an \emph{Euler vector field}.

The notion of a double vector bundle \cite{Pradines:1974} (or a higher $n$-tuple vector bundle)   is conceptually  clear in the graded language in terms of mutually commuting homogeneity structures, or equivalently weight vector f\/ields; see~\cite{Grabowski:2006,Grabowski2012}. This leads to the higher analogues known as \emph{$n$-tuple graded bundles}, which are  manifolds for which  the structure sheaf carries an $\mathbb{N}^{n}$-grading such that all the weight vector f\/ields are h-complete and pairwise commuting.  In particular a \emph{double graded bundle} consists of a manifold and a pair of mutually commuting and h-complete weight vector f\/ields. If all the weights are either zero or one, then we speak of an \emph{$n$-tuple vector bundle}.

\subsection[$Q$-manifolds and related structures]{$\boldsymbol{Q}$-manifolds and related structures}
Let us brief\/ly def\/ine some structures on supermanifolds that we will employ throughout this paper. The main purpose is to set some nomenclature and notation. The graded versions will be fundamental throughout this paper.

\begin{Definition}
An odd vector f\/ield $Q  \in \Vect(\mathcal{M})$ on a supermanifold $\mathcal{M}$ is said to be a~\emph{homological vector field}  if and only if $2 Q^{2} =  [Q,Q]=0$. Note that, as we have an odd vector f\/ield, this condition is generally non-trivial. A pair $(\mathcal{M},Q)$, where $\mathcal{M}$ is a supermanifold and $Q \in \Vect(\mathcal{M})$ is called a~\emph{$Q$-manifold}. A~\emph{morphism of $Q$-manifolds} is a morphism of supermanifolds that  relates the respective homological vector f\/ields.
\end{Definition}

The nomenclature `$Q$-manifold' is due to Schwarz \cite{Schwarz:1993}.  As is well known, $Q$-manifolds, equipped with an additional grading, give a very economical description of Lie algebroids as f\/irst discovered by  Va$\breve{\textrm{{\i}}}$ntrob \cite{Vaintrob:1997}. Recall the standard notion of a Lie algebroid as a vector bundle $E \rightarrow M$ equipped with a Lie bracket on the sections $[\bullet, \bullet]\colon  \Sec(E) \times \Sec(E) \rightarrow \Sec(E)$ together with an anchor $\rho\colon  \Sec(E) \rightarrow \Vect(M)$ that satisfy the Leibniz rule
\begin{gather*}
 [u,fv] = \rho(u)[f]   v  +   f [u,v],
\end{gather*}

\noindent for all $u,v \in \Sec(E)$ and $f \in C^{\infty}(M)$. The Leibniz rule implies that the anchor is actually a Lie algebra morphism: $\rho\left([u,v]\right) = [\rho(u), \rho(v)]$ (see, e.g.,~\cite{Grabowski:2003}). If we pick some local basis for the sections~$(e_{a})$, then the structure functions of a Lie algebroid are def\/ined by
\begin{gather*}
[e_{a} , e_{b}]  =  C_{ab}^{c}(x)e_{c}, \qquad
\nonumber \rho(e_{a})  =  \rho_{a}^{A}(x)\frac{\partial}{\partial x^{A}},
\end{gather*}
 where we have local coordinates $(x^{A})$ on~$M$. These structure functions satisfy some compatibility conditions which can be neatly encoded in a homological vector f\/ield of weight one on $\Pi E$. Let us employ local coordinates
\begin{gather*}
\big\{\underbrace{x^{A}}_{0}, \, \underbrace{\xi^{a}}_{1}\big\},
\end{gather*}
on $\Pi E$. Here the weight of coordinates corresponds to the natural weight induced by the homogeneity structure associated with the vector bundle structure \cite{Grabowski:2006} and this also corresponds to the Grassmann parity. We just brief\/ly remark that one can also def\/ine Lie algebroids in the category of supermanifolds where it natural to consider the weight and Grassmann parity as being independent. The homological vector f\/ield encoding the Lie algebroid is
\begin{gather*}
Q = \xi^{a}\rho_{a}^{A}(x) \frac{\partial}{\partial x^{A}} + \frac{1}{2}\xi^{a}\xi^{b}C_{ba}^{c}(x)\frac{\partial}{\partial \xi^{c}}.
\end{gather*}
The axioms of a Lie algebroid are then equivalent to $2Q^{2} = [Q,Q]=0$. We will, by minor abuse of nomenclature, also refer to the graded $Q$-manifold $(\Pi E, Q)$ as a  Lie algebroid.

A morphism of Lie algebroid is then understood as a morphism of graded $Q$-manifolds. That is, we have a morphism of super graded bundles that relates the respective homological vector f\/ields. Note that the def\/inition of a Lie algebroid morphism in terms of $Q$-manifold morphisms  is  equivalent to the less obvious notion of a morphism of Lie algebroids as described in~\cite{Mackenzie2005}.

\begin{Example}
Any Lie algebra $(\mathfrak{g}, [\,,\,])$ can be encoded in a homological vector f\/ield on the linear supermanifold $\Pi \mathfrak{g}$. Let us employ local coordinates $(\xi^{a})$ on $\Pi \mathfrak{g}$; then we have
\begin{gather*}
Q = \frac{1}{2}\xi^{a}\xi^{b}C_{ba}^{c} \frac{\partial}{\partial \xi^{c}},
\end{gather*}
\noindent where $C_{ba}^{c}$ is the structure constant of the Lie algebra in question.
In this case the base manifold is just a point and the anchor map is trivial.
\end{Example}

\begin{Example}
In the other extreme, the tangent bundle of a manifold~$\sT M$ is naturally a Lie algebroid  for which the anchor is the identity map. The homological vector on $\Pi \sT M$ that encodes the Lie algebroid structure is nothing other than the de Rham dif\/ferential.
\end{Example}

\begin{Definition}
 An odd Hamiltionain $S \in C^{\infty}(\sT^{*}\mathcal{M})$  that is quadratic in momenta (i.e., f\/ibre coordinates) is said to be a \emph{Schouten structure} if and only if $\{S,S\} = 0$, where the bracket is the canonical Poisson bracket on the cotangent bundle. A pair $(\mathcal{M},S)$, where $\mathcal{M}$ is a supermanifold and $S \in C^{\infty}(\sT^{*}{\mathcal M})$ is a Schouten structure, is called a \emph{$S$-manifold}.  The associated \emph{Schouten bracket} on $C^{\infty}(M)$ is given  as a derived bracket
\begin{gather*}
\SN{f,g}_{S} :=  \{\{f, S\},g \}.
\end{gather*}
\end{Definition}

We must remark here that non-trivial Schouten structures only exist on supermanifolds. A~Schouten bracket is also known as an odd Poisson bracket and satisf\/ies the appropriate graded versions of the Jacobi identity and Leibniz rule.

\begin{Example}
A Lie algebroid structure on a vector bundle $E \rightarrow M$ can be encoded in a weight one  Schouten structure on the supermanifold $\sT^{*}\Pi E^{*}$. Let us employ natural local coordinates  on $\sT^{*}\Pi E^{*}$ (with indicated bi-degrees)
\begin{gather*}
\big(\underbrace{x^{A}}_{(0,0)}, \, \underbrace{\chi_{a}}_{(0,1)}, \, \underbrace{p_{B}}_{(1,1)}, \, \underbrace{\theta^{b}}_{(1,0)}\big) ,
\end{gather*}
Then, the Schouten structure encoding a Lie algebroid is given by
\begin{gather*}
S = \theta^{a}\rho_{a}^{A}(x) p_{A} + \frac{1}{2}\theta^{a}\theta^{b}C_{ba}^{c}(x) \chi_{c}
\end{gather*}
and the axioms of a Lie algebroid are equivalent to $\{ S,S\} =0$, where the bracket here is the canonical Poisson bracket on the cotangent bundle.
\end{Example}

\begin{Definition}
A supermanifold $M$ is said to be a \emph{$QS$-manifold} if it is simultaneously a $Q$-manifold and a $S$-manifold such that $L_{Q}S := \{\mathcal{Q}, S \} =0$, where $\mathcal{Q} \in C^{\infty}(\sT^{*}M)$ is the symbol of the homological vector f\/ield~$Q$.
\end{Definition}

As the symbol map sends the Lie bracket of vector f\/ields to the Poisson bracket, a $QS$-manifold can be considered as a supermanifold equipped with odd Hamiltonians, one linear and one quadratic in momenta, that satisfy $\{\mathcal{Q}, \mathcal{Q} \} =0$,  $\{S,S\}=0$ and $\{ \mathcal{Q}, S\} =0$. The def\/inition of a $QS$-manifold goes back to Voronov~\cite{Voronov:2001qf}.

Graded $QS$-manifolds give us a convenient way to understand Lie bi-algebroids, which is due to Voronov~\cite{Voronov:2001qf}, but also see  Kosmann-Schwarzbach \cite{Kosmann-Schwarzbach:1995}  who modif\/ied the original def\/inition of Mackenzie and  Xu~\cite{Mackenzie:1994}. The original def\/inition involved the dif\/ferential associated with the dual Lie algebroid and Lie bracket on sections of the Lie algebroid, and not the associated  Schouten bracket on `multivector f\/ields'.  Following  Kosmann-Schwarzbach,  a Lie bi-algebroid is a pair of Lie algebroids $(E, E^{*})$ such that
\begin{gather*}
Q_{E}\SN{X,Y}_{E^{*}} = \SN{Q_{E}X, Y}_{E^{*}} + (-1)^{\widetilde{X}}\SN{X, Q_{E} Y}_{E^{*}},
\end{gather*}
for all `multivector f\/ields' $X$ and $Y \in C^{\infty}(\Pi E)$. That is the homological vector f\/ield encoding the Lie algebroid structure on $E$ must be a derivation with respect to the Schouten bracket that encodes the Lie algebroid structure on the dual vector bundle $E^{*}$. It is not hard to see that this def\/inition is equivalent to the compatibility condition
\begin{gather*}
\mathcal{L}_{Q_{E}}S_{E^{*}} = \{ \mathcal{Q}_{E}, S_{E^{*}}\} =0,
\end{gather*}
and so we have a $QS$-manifold.
 Here we use the canonical isomorphism $\sT^{*}\Pi E \simeq \sT^{*}\Pi E^{*}$.
In natural local coordinates the symbol of the homological vector f\/ield and the Schouten structure are given by
\begin{gather*}
\mathcal{Q}_{E}  =  \theta^{a}\rho_{a}^{A}(x)p_{A} + \frac{1}{2} \theta^{a}\theta^{b}C_{ba}^{c}(x) \chi_{c},\qquad
S_{E^{*}}  = \chi_{a}\rho^{aA}(x)p_{A} + \frac{1}{2} \chi_{a}\chi_{b} C^{ba}_{c}(x) \theta^{c},
\end{gather*}
 which are clearly of bi-weight $(2,1)$ and $(1,2)$ respectively. We can then \emph{define} a Lie bi-algebroid as the graded $QS$-manifold $(\sT^{*}\Pi E, Q_{E}, S_{E^{*}})$.

The above def\/inition of a Lie bi-algebroid is not manifestly symmetric in $E$ and $E^{*}$, however, due to the isomorphism
\begin{gather*}
 (\sT^{*}\Pi E,Q_E,S_{E^*} ) \simeq  (\sT^{*}\Pi E^{*},S_E,Q_{E^*} ) ,
\end{gather*}
it is clear that if $(E, E^{*})$ is a Lie bi-algebroid, then so is $(E^{*}, E)$.

\subsection{Weighted Lie algebroids}
 One can think of  a \emph{weighted Lie algebroid} as a Lie algebroid in the category of graded bundles or, equivalently, as a graded bundle in the category of Lie algebroids~\cite{Bruce:2014}. Thus one should think of weighted Lie algebroids as Lie algebroids that carry an additional compatible grading. For the purposes of this paper, we will def\/ine weighted Lie algebroids here using the notion of homogeneity structures as this will turn out to be a useful point of view when dealing with the associated Lie theory.  Let us recall the def\/inition of a graded-linear bundle, which is fundamental in the notion of a weighted Lie algebroid.

\begin{Definition}
A manifold $D_{(k-1, 1)}$ equipped with a pair of homogeneity structures $(\widehat{\rmh}, \widehat{\rml})$ of degree $k-1$ and $1$ respectively is called a \emph{graded-linear bundle} of degree $k$, which we will abbreviate as $\mathcal{GL}$-bundle, if and only if the respective actions commute.
\end{Definition}

The above def\/inition is equivalent to the def\/inition given in \cite{Bruce:2014} in terms of mutually commuting h-complete weight vector f\/ields of degree $k-1$ and $1$. In all, by passing to total weight, a~$\mathcal{GL}$-bundle is a graded bundle of degree $k$. We will denote the base def\/ined by the vector bundle structure as $B_{k-1} \rightarrow M$. We will  generally employ the shorthand notation~$D_{k}$ for a~$\mathcal{GL}$-bundle~$D_{(k-1,1)}$ of degree~$k$ in this paper.

\begin{Example}[\cite{Bruce:2014}]
Let $F_{k-1}$ be a graded bundle of degree $k-1$, then $\sT F_{k-1}$ is canonically a~$\mathcal{GL}$-bundle of degree~$k$. The degree $k-1$ homogeneity structure is simply the tangent lift of the homogeneity structure on the initial graded bundle, while the degree one homogeneity structure is that associated with the natural vector bundle structure of the tangent bundle. That is, if we equip $F_{k-1}$ with local coordinates $(x^{\alpha}_{u})$, here $0 \leq u < k$, then  $\sT F_{k-1}$ can be naturally equipped with coordinates
\begin{gather*}
\big(\underbrace{x^{\alpha}_{u}}_{(u,0)}, \, \underbrace{\delta x_{u+1}^{\beta}}_{(u,1)}\big).
\end{gather*}
\end{Example}

\begin{Example}[\cite{Bruce:2014}]
Similarly to the above example, if $F_{k-1}$ be a graded bundle of degree $k-1$, then $\sT^{*} F_{k-1}$ is canonically a $\mathcal{GL}$-bundle of degree~$k$. Note there we have to employ the \emph{phase lift} of the homogeneity structure on $F_{k-1}$ to ensure that we do not leave the category of graded bundles~\cite{Grabowski2013}; that is we do not want negative weight coordinates. This amounts to an appropriate shift in the weight and allows us to employ homogeneous local coordinates of the form
\begin{gather*}
\big(\underbrace{x^{\alpha}_{u}}_{(u,0)}, \, \underbrace{p_{\beta}^{k-u}}_{(k-1-u, 1)}\big).
\end{gather*}
\end{Example}
Note that as we have a linear structure on $\mathcal{GL}$-bundles, that is we have a homogeneity structure of weight one, or equivalently an Euler vector f\/ield, applying the parity reversion functor makes sense. Thus, we can consider the $\mathcal{GL}$-antibundle $(\Pi D_{k}, \Pi\,\widehat{\rmh}, \Pi\, \widehat{\rml})$, where we def\/ine viz
\begin{gather*}
\Pi \widehat {\rmh}_{t}\colon \ \Pi D_{k} \rightarrow \Pi D_{k},
\end{gather*}
\noindent and similarly for $\Pi  \widehat{\rml}$. We can now def\/ine a weighted Lie algebroid as follows;

\begin{Definition}
A \emph{weighted Lie algebroid} of degree $k$ is the quadruple
\begin{gather*}
\big(\Pi D_{k},\Pi \widehat{\rmh}, \Pi  \widehat{\rml}, Q\big ),
\end{gather*}
 where $(\Pi D_{k}, Q)$ is a $Q$-manifold and  $(\Pi D_{k},\Pi \widehat{\rmh}, \Pi \widehat{\rml})$ is a $\mathcal{GL}$-antibundle such that
\begin{gather}
\label{xx} Q \circ \big(\Pi \widehat{\rmh}_{t}\big)^{*}  =  \big(\Pi\,\widehat{\rmh}_{t}\big)^{*} \circ Q,\\
s   Q \circ \big(\Pi  \widehat{\rml}\big)_{s})^{*}  =  \big(\Pi \widehat{\rml}_{s}\big)^{*} \circ Q,\nonumber
\end{gather}
for all $t$  and all $s \in \mathbb{R}$.
\end{Definition}

The above def\/inition is really just the statement that a weighted Lie algebroid can be def\/ined as a $\mathcal{GL}$-antibundle equipped with a homological vector of bi-weight $(0,1)$.
By passing to total weight,  weighted Lie algebroids produce examples of higher Lie algebroids, cf.\ Voronov~\cite{voronov-2010}.

 \begin{Example}It is not hard to see that a weighted Lie algebroid of degree~$1$ is just a standard Lie algebroid; the additional homogeneity structure is trivial and so we have a graded super bundle of degree one and a weight one homological vector f\/ield.
 \end{Example}

 \begin{Example}
 Similarly to the previous example, weighted Lie algebroids of degree $2$  (i.e., of bi-degree $(1,1)$) are $\mathcal{VB}$-algebroids; we now have a~double super vector bundle structure and a~weight $(0,1)$ homological vector f\/ield \cite{Gracia-Saz:2009}.
 \end{Example}
Further natural examples of weighted Lie algebroids include the tangent bundle of a graded bundle and the higher order tangent bundles of a Lie algebroid, see \cite{Bruce:2014}.

 The above def\/inition allows for a further more economical def\/inition via the following proposition, which follows immediately from (\ref{xx}).

\begin{Proposition}
A weighted Lie algebroid of degree $k$ can equivalently be defined as a Lie algebroid $(\Pi E, Q)$ equipped with a homogeneity structure of degree $k-1$ such that
\begin{gather*}
\Pi  \widehat{\rmh}_{t}\colon \  \Pi E \rightarrow \Pi E,
\end{gather*}
\noindent is a Lie algebroid morphism for all $t \in \mathbb{R}$.
\end{Proposition}

\begin{Remark}
Bursztyn, Cabrera and del Hoyo  \cite[Theorem~3.4.3]{Bursztyn:2014} establish  that $\mathcal{VB}$-algebroids can be def\/ined as Lie algebroids equipped with a compatible regular homogeneity structure. Note that they make explicit use of the Poisson structure associated with a Lie algebroid, where we prefer the (equivalent) description in terms of $Q$-manifolds.  Their motivation  for describing $\mathcal{VB}$-algebroids in terms of compatible homogeneity structures  is, much like our motivation,  to tackle the associated Lie theory.
\end{Remark}

To get a more traditional description of a weighted Lie algebroid, let us employ homogeneous local coordinates
\begin{gather*}
\big(\underbrace{x^{\alpha}_{u}}_{(u,0)}, \, \underbrace{\theta^{I}_{u+1}}_{(u,1)}\big),
\end{gather*}
on $\Pi D_{k}$, where  $0 \leq u \leq k-1$ and this label refers to the total weight. Note that the second component of the bi-weight encodes the Grassmann parity of the coordinates; that is the~$\theta$'s are anticommuting coordinates. In these homogeneous local coordinates the homological vector f\/ield encoding the structure of a weighted Lie algebroid is
\begin{gather*}
Q = \theta^{I}_{u-u'+1}Q_{I}^{\alpha}[u'](x) \frac{\partial}{\partial x^{\alpha}_{u} }  + \frac{1}{2} \theta^{I}_{u''-u'+1} \theta^{J}_{u-u''+1}Q_{JI}^{K}[u'](x)\frac{\partial}{\partial \theta_{u+1}^{K} } ,
\end{gather*}
 where $Q_{I}^{\alpha}[u']$ and $Q_{IJ}^{K}[u']$ are the homogenous parts of the structure functions of degree~$u'$. In the notation employed here, any  $\theta$ with seemingly zero or negative total weight  are set to zero.

In \cite{Bruce:2014} homogeneous  linear sections of $D_{k} \rightarrow B_{k-1}$ of weight $r$ were def\/ined functions on the $\mathcal{GL}$-bundle  $D_{k}^{*}$ of bi-weight $(r-1, 1)$. Let us equip $D_{k}^{*}$ with homogeneous local coordinates
\begin{gather*}
\big(\underbrace{x^{\alpha}_{u}}_{(u,0)}, \, \underbrace{\pi_{I}^{u+1}}_{(u,1)}\big).
\end{gather*}
  In these local coordinates a linear section of $D_{k}$ is given by
\begin{gather*}
s_{r}= s_{r-u-1}^{I}(x)\pi_{I}^{u+1}.
\end{gather*}
 Such linear  sections can also be understood as vertical vector f\/ields constant along the f\/ibres  of~$\Pi D_{k}$. In the graded language we consider vector f\/ields that are of weight~$-1$ with respect to the second component of the bi-weight on~$\Pi D_{k}$.  Note that we have a shift in the bi-weight and also the Grassmann parity.  The reason we use~$\Pi D_{k}$ rather than $D_{k}$ will become clear in a moment when we consider how to construct the Lie algebroid bracket and the anchor. By employing  the homogeneous local coordinates on~$\Pi D_{k}$ introduced perviously we have the identif\/ication
\begin{gather*}
s_{r} \leftrightsquigarrow i_{s_{r}}:= s^{I}_{r-u-1}(x)\frac{\partial}{\partial \theta_{k-u}^{I}} ,
\end{gather*}
 which is really no more than the `interior product' generalised to sections of a vector bundle. We do not require any extra structure here for this identif\/ication, just the vector bundle structure is required. Note the shift in the bi-weight $(r-1,1)\rightarrow (r-k,-1)$, or in other words we have a~shift  of $(-(k-1), -2)$. This shift will be very important.

We can now employ the derived bracket formalism~\cite{Kosmann-Schwarzbach:2004} to construct the Lie algebroid bracket;
\begin{gather*}
i_{[s_{1}, s_{2}]} = [[Q, i_{s_{1}}], i_{s_{2}}],
\end{gather*}
 up to a possible overall minus sign by convention. By inspection we see that $i_{[s_{1}, s_{2}]}$ is of bi-weight $(r_{1}+ r_{2}- 2k, -1)$, where $s_{1}$ is a linear section of degree~$r_{1}$ and similarly for~$s_{2}$. This  means that the Lie algebroid bracket $[s_{1},s_{2}]$ is of bi-weight $(r_{1} + r_{2}-1 -k, 1)$; that is the Lie algebroid bracket carries weight~$-k$.

\begin{Definition}
A \emph{morphism of weighted Lie algebroids} is a morphism of the underlying Lie algebroids (understood as a morphism of $Q$-manifolds) that intertwines the respective additional homogeneity structures.
\end{Definition}

 In particular,  there is a forgetful functor from the category of weighted Lie algebroids to Lie algebroids given by forgetting the homogeneity structure. The image of this forgetful functor is  not a full subcategory of the category of Lie algebroids.  We will denote the category of Lie algebroids as $\catname{Algd}$ and the subcategory of weighted Lie algebroids as~$\catname{wAlgd}$.

 \begin{Proposition}\label{prop:algebroid tower}
 If $(\Pi D_{k}, Q)$ is weighted Lie algebroid of degree $k$, then each $\Pi D_{j}$ for $0\leq j < k$ comes with the structure of a weighted Lie algebroid of degree $j$. In particular, $\Pi E = \Pi D_{1}$ is a `genuine' Lie algebroid. Moreover,  the projection form any `higher level' to a `lower level' is a morphism of weighted Lie algebroids.
  \end{Proposition}

 \begin{proof}
 Because of the condition that the weight of $Q$ is $(0,1)$ and that there are no negatively graded coordinates, we know that $Q$ is projectable to any of the `levels' of the af\/f\/ine f\/ibrations
 \begin{gather}\label{eqn:levels}
 \Pi D_{(k-1,1)} \rightarrow \Pi D_{(k-2, 1)} \rightarrow \cdots \rightarrow \Pi D_{(1,1)} \rightarrow \Pi D_{(0,1)}.
  \end{gather}
  \noindent This can be directly checked using the local expression for the homological vector f\/ield. Moreover, the projection of~$Q$ to any level is a~homological vector f\/ield of weight~$(0,1)$ and hence we have  the structure of a weighted Lie algebroid. The fact that the projections give rise to weighted Lie algebroid morphisms follows directly.
 \end{proof}

\subsection{Lie groupoids}

Here we shall set some notation and recall some well-known results. Let $\mathcal{G} \rightrightarrows M$ be an arbitrary Lie groupoid with \emph{source map} $\underline{s}\colon \mathcal{G} \rightarrow M$ and \emph{target map} $\underline{t}\colon \mathcal{G} \rightarrow M$. There is also the inclusion map $\iota \colon M \rightarrow \mathcal{G}$   and a \emph{partial multiplication}  $(g,h) \mapsto g\cdot h$ which is def\/ined on $\mathcal{G}^{(2)}=\{(g,h)\in\mathcal{G}\times\mathcal{G}\colon  \underline{s}(g) = \underline{t}(h)\}$. Moreover, the manifold $\mathcal{G}$ is foliated by $\underline{s}$-f\/ibres $\mathcal{G}_{x}= \{  g \in \mathcal{G}\,|\, \underline{s}(g) =x\}$, where $x \in M$. As by def\/inition the source and target maps are submersions, the $\underline{s}$-f\/ibres are themselves smooth manifolds. Geometric objects associated with this foliation will be given the superscript $\underline{s}$. In particular, the distribution tangent to the leaves of the foliation will be denoted by $\sT^{\underline{s}} \mathcal{G}$. To ensure no misunderstanding with the notion of a Lie groupoid morphism we recall the def\/inition we will be using.

\begin{Definition}\label{def:Lie algebroid Morphism}
Let $\mathcal{G}_{i} \rightrightarrows  M_{i}$ $(i=1,2)$ be a pair of Lie groupoids. Then a \emph{Lie groupoid morphisms} is a pair of maps $(\Phi, \phi)$ such that the following diagram is commutative
\begin{gather*}
\begin{xy}
(0,20)*+{\mathcal{G}_{1}}="a"; (20,20)*+{\mathcal{G}_{2}}="b";%
(0,0)*+{M_{1}}="c"; (20,0)*+{M_{2}}="d";%
{\ar "a";"b"}?*!/_2mm/{\Phi};
{\ar@<1.ex>"a";"c"} ;
{\ar@<-1.ex> "a";"c"} ?*!/^3mm/{s_{1}} ?*!/_6mm/{t_{1}};
{\ar@<1.ex>"b";"d"};%
{\ar@<-1.ex> "b";"d"}?*!/^3mm/{s_{2}} ?*!/_6mm/{t_{2}};  %
{\ar "c";"d"}?*!/^3mm/{\phi};
\end{xy}
\end{gather*}
in the sense that $\underline{s}_{2}\circ \Phi = \phi \circ \underline{s}_{1}$ and $\underline{t}_{2}\circ \Phi = \phi \circ \underline{t}_{1} $,
 subject to the further condition that $\Phi$ respect the (partial) multiplication; if $g,h \in \mathcal{G}_{1}$ are composable, then  $ \Phi(g \cdot h) = \Phi(g) \cdot \Phi(h)$. It then follows that for $x \in M_{1}$  we have $\Phi(\Id_{x}) = \Id_{\phi(x)}$ and  $\Phi(g^{-1}) = \Phi(g)^{-1}$.
\end{Definition}

We will denote the category of Lie groupoids as $\catname{Grpd}$. It is well known that via a dif\/fe\-ren\-tia\-tion procedure one can construct the \emph{Lie functor}
\begin{gather*}
\catname{Grpd} \stackrel{\Lie}{\xrightarrow{\hspace*{30pt}}} \catname{Algd},
\end{gather*}
that sends a Lie groupoid to its Lie algebroid, and sends morphisms of Lie groupoids to morphisms of the corresponding Lie algebroids. However, as is also well known, we do not have an equivalence of categories as not all Lie  algebroids arise as the inf\/initesimal versions of Lie groupoids.  There is no direct generalisation of Lie~III,  apart from the local case. The obstruction to the integrability of Lie algebroids, the so called \emph{monodromy groups},  were f\/irst uncovered by Crainic and Fernandes \cite{Crainic:2003}.  For the transitive case the topological obstruction was uncovered by Mackenzie \cite{Mackenzie:1987}.  The interested reader can consult the original literature or the very accessible lecture notes also by Crainic and  Fernandes~\cite{Crainic:2011}. To set some notation and nomenclature, given a  Lie groupoid~$\mathcal{G}$, we say that $\Lie(\mathcal{G}) = \rmA(\mathcal{G})$ \emph{integrates}  $\rmA(\mathcal{G})$. Moreover, if $\Phi\colon \mathcal{G} \rightarrow \mathcal{H}$ is a~morphism of Lie groupoids, then we will write $\Phi' = \Lie(\Phi) \colon  \rmA(\mathcal{G}) \rightarrow \rmA(\mathcal{H})$ for the corresponding Lie algebroid morphism,  which actually comes from the dif\/ferential $\sT\Phi\colon \sT\mathcal{G} \rightarrow \sT\mathcal{H}$ restricted to $\underline{s}$-f\/ibres at the submanifold~$M$.

Let  us just recall Lie II theorem, as we will need it later on.

\begin{Theorem}[Lie II] \label{thm:integration morphisms}
Let $\mathcal{G} \rightrightarrows M$ and $\mathcal{H}\rightrightarrows N$ be Lie groupoids. Suppose that $\mathcal{G}$ is source simply-connected and that $\phi\colon \rmA(\mathcal{G}) \rightarrow \rmA(\mathcal{H})$ is a Lie algebroid morphism between the associated Lie algebroids. Then, $\phi$ integrates to a unique Lie groupoid morphisms $\Phi\colon \mathcal{G} \rightarrow \mathcal{H}$ such that $\Phi' = \phi$.
\end{Theorem}

This generalisation  of Lie~II to the groupoid case was f\/irst proved by Mackenzie ans Xu~\cite{Mackenzie:2000}. A simplif\/ied proof was obtained shortly after by Moerdijk and  Mr\v{c}un~\cite{Moerdijk:2002}.  Note that the assumption that the Lie groupoid~$\mathcal{G}$ must be source simply-connected is essential.

\section{Weighted Lie groupoids}\label{sec:weighted Lie groupoids}
To avoid any possible confusion with the overused adjective `graded', following the nomenclature in~\cite{Bruce:2014} we will refer to weighted rather than graded Lie groupoids.  The term  `graded groupoid' appears in the work of Mehta~\cite{Mehta:2006} as groupoids in the category of $\mathbb{Z}$-graded manifolds;  we will comment on this further. In this work we will content ourselves with working in the strictly commutative (in opposition to graded-commutative) setting. The basic idea is to take the def\/inition of a Lie groupoid and replace `smooth manifold' everywhere with `graded bundle'. To do this economically, we will make use of the description of graded bundles in terms of smooth manifolds equipped with a homogeneity structure of degree $k$,  so that morphisms of graded bundles are just smooth maps intertwining the respective homogeneity structures. In essence, we require that  the Lie groupoid structure functions respect the graded structure which is encoded in the homogeneity structure.

\subsection{The def\/inition of weighted Lie groupoids via homogeneity structures}
Let us formalise the comments in the beginning of this section with the following def\/inition;

\begin{Definition}
A \emph{weighted Lie groupoid} of degree $k$ is a Lie groupoid $\Gamma_{k} \rightrightarrows B_{k}$, together with a  homogeneity structure  $\underline{\rmh}\colon  \mathbb{R} \times \Gamma_{k} \rightarrow \Gamma_{k}$ of degree~$k$, such that $\underline{\rmh}_{t}$ is a Lie groupoid morphism for all $t \in \mathbb{R}$. Such homogeneity structures will be referred to as \emph{multiplicative homogeneity structures}.
\end{Definition}

\begin{Remark}
Equivalently one could think of a weighted Lie groupoid in terms of multiplicative weight vector f\/ields. This follows from the def\/inition of a weighted Lie groupoid and the fact that weight vector f\/ields associated with homogeneity structures are complete. However, it will be convenient to use the homogeneity structures rather than weight vector f\/ields from the perspective of Lie theory.
\end{Remark}

Let us unravel some of the structure here. First, as both the source and  target maps are submersions and by assumption $\Gamma_{k}$ is a graded bundle of degree $k$, $B_{k}$ is also a graded bundle of degree~$k$ (we will consider lower degree graded bundles to be included as higher degree graded bundles if necessary).  Secondly, it will be convenient to consider the homogeneity structure~$\underline{\rmh}$ as a pair of structures: $\underline{\rmh}_{t} = (\rmh_{t},\rmg_{t})$. Then, consider the  following commutative diagram
\begin{gather*}
\begin{xy}
(0,20)*+{\Gamma_{k}}="a"; (20,20)*+{\Gamma_{k}}="b";%
(0,0)*+{B_{k}}="c"; (20,0)*+{B_{k}}="d";%
{\ar "a";"b"}?*!/_3mm/{\rmh_{t}};
{\ar@<1.ex>"a";"c"} ;
{\ar@<-1.ex> "a";"c"} ?*!/^3mm/{\underline{s}} ?*!/_6mm/{\underline{t}};
{\ar@<1.ex>"b";"d"};%
{\ar@<-1.ex> "b";"d"}?*!/^3mm/{\underline{s}} ?*!/_6mm/{\underline{t}};  %
{\ar "c";"d"}?*!/^3mm/{\rmg_{t}};
\end{xy}
\end{gather*}
 which means that
 $\underline{s}\circ \rmh_{t} = \rmg_{t} \circ \underline{s}$,  and $\underline{t}\circ \rmh_{t} = \rmg_{t} \circ \underline{t}$.
That is, the source and target maps intertwine the respective homogeneity structure.
 Furthermore, we have $\rmh_{t}(g \cdot h) = \rmh_{t}(g) \cdot \rmh_{t}(h)$, meaning that the homogeneity structure respects the (partial) multiplication, or indeed vice versa. From the def\/inition of a Lie groupoid morphism, compatibility with the identities and inverses follows. In short, the graded structure is fully compatible with the Lie groupoid structure. Or, in other words, we have a  Lie groupoid object in the category of graded bundles or equivalently vice-versa.

\begin{Definition}
A \emph{morphism of weighted Lie groupoids} is a Lie groupoid morphism that intertwines the respective homogeneity structures.
\end{Definition}

In other words, a morphism of weighted Lie groupoids must in addition to being a morphism of Lie groupoids be a morphism of graded bundles. Clearly we have a forgetful functor from the category of weighted Lie groupoids to the category of Lie groupoids which forgets the homogeneity structure. Similar to the case of algebroids, the image of this forgetful functor is  not a~full subcategory of the category of Lie groupoids. We will denote the category of weighted Lie groupoids as $\catname{wGrpd}$.

\begin{Remark}
 Graded groupoids in the sense of Mehta \cite{Mehta:2006,Mehta:2009} are understood as groupoid objects in the category of $\mathbb{Z}$-graded manifolds.  The $\mathbb{Z}_{2}$-grading is a consequence of the  $\mathbb{Z}$-grading and so $\mathbb{Z}$-manifolds are not quite as general as Voronov's graded supermanifolds (cf.~\cite{Voronov:2001qf}).  Without details,  a graded groupoid is a pair of $\mathbb{Z}$-graded manifolds  $(\mathcal{G}, M)$, equipped with surjective submersions $s,t \colon \mathcal{G} \rightrightarrows M$ (the source and target), together with a partial multiplication map, an identity  map and an inverse map which all satisfy the appropriate commutative diagrams to def\/ine a groupoid structure. To be very clear, all the structural maps are morphisms in the category of $\mathbb{Z}$-graded manifolds.   One should of course  be reminded of the  def\/inition of Lie supergroups as group objects in the category of supermanifolds and indeed Mehta's constructions cover the case of Lie supergroupoids.

 From the \emph{definition} of a weighted Lie groupoid one can check directly that all the structural maps are morphisms in the category of graded bundles; all the maps relate the respective weight vector f\/ields (or equivalently the actions of the respective homogeneity structures). Thus, one could start with a def\/inition of a weighted Lie groupoid  in terms closer to that used by Mehta to def\/ine his graded groupoids. However, in light of Lie theory it will turn out to be very convenient to describe weighted Lie groupoids as Lie groupoids with a compatible homogeneity structure. Mehta by working in the category of $\mathbb{Z}$-graded manifolds did not have anything like a~homogeneity structure at his disposal. While it is straight forward to show that dif\/ferentiating graded groupoids leads to graded algebroids, the question of integration seems less tractable.
\end{Remark}

\subsection[Relation with $\mathcal{VB}$-groupoids]{Relation with $\boldsymbol{\mathcal{VB}}$-groupoids}

Bursztyn, Cabrera and del Hoyo  (cf.~\cite[Theorem~3.2.3]{Bursztyn:2014}) establish a one-to-one correspondence between $\mathcal{VB}$-groupoids, which have several equivalent def\/initions (see for example~\cite{Gracia-Saz:2010}), and Lie groupoids that have a regular action of $(\mathbb{R}, \cdot)$ as Lie groupoid morphisms. In essence, they show the equivalence between $\mathcal{VB}$-groupoids and weighted Lie groupoids of degree one. We take the attitude that the `correct def\/inition'  of a~$\mathcal{VB}$-groupoid is implied by this correspondence. We consider the double structure, that is the pair of Lie groupoids and the pair of vector bundles together with a collection of not-so obvious compatibility conditions,  as being derived  rather than fundamental.  This helps motivate our def\/inition of weighted Lie groupoids using homogeneity structures.

\begin{Proposition}
For any weighted Lie groupoid $\Gamma_{k} \rightrightarrows B_{k}$ there is an underlying `genuine' or `ungraded'  Lie groupoid $\mathcal{G} \rightrightarrows M$ described by the following surjective groupoid morphism
\begin{gather*}
\begin{xy}
(0,20)*+{\Gamma_{k}}="a"; (20,20)*+{\mathcal{G}}="b";%
(0,0)*+{B_{k}}="c"; (20,0)*+{B_{0}}="d";%
{\ar "a";"b"}?*!/_3mm/{\rmh_{0}};
{\ar@<1.ex>"a";"c"} ;
{\ar@<-1.ex> "a";"c"} ?*!/^3mm/{\underline{s}} ?*!/_6mm/{\underline{t}};
{\ar@<1.ex>"b";"d"};%
{\ar@<-1.ex> "b";"d"}?*!/^3mm/{\underline{\sigma}} ?*!/_6mm/{\underline{\tau}};  %
{\ar "c";"d"}?*!/^3mm/{\rmg_{0}};
\end{xy}
\end{gather*}
 Moreover, if $\Gamma_{k}$ is source simply-connected, then so is $\mathcal{G}$.
\end{Proposition}

\begin{proof}
By construction we have projections $ \rmh_{0}\colon \Gamma_{k} \rightarrow \Gamma_{0} := \mathcal{G}$ and $\rmg_{0}\colon B_{k} \rightarrow B_{0}:= M$. Thus $\mathcal{G} \subset \Gamma_{k}$ and $M \subset \mathcal{G}$ are embedded manifolds. Clearly $\mathcal{G}$ and $M$ def\/ine a (set theoretical) subgroupoid of $\Gamma_{k} \rightrightarrows B_{k}$. As $(\rmh_{t}, \rmg_{t})$ is a morphism of Lie groupoids for all $t \in \mathbb{R}$, including zero, it follows that the source map $\underline{\sigma}\colon \mathcal{G} \rightarrow M$ is a submersion. Thus we have the structure of a~Lie groupoid.  Now consider any point $p \in \mathcal{G}$. As $\rmh_{0}(p) = p$ the we have an induced retraction map $\underline{s}^{-1}(p) \rightarrow \underline{\sigma}^{-1}(p)$ and thus we have a homomorphisms between the respective fundamental groups. Thus if~$\Gamma_{k}$ is source simply-connected, then so is~$\mathcal{G}$.
\end{proof}

Actually, slightly modifying the above proof, we can derive the groupoid version of the tower of Lie algebroid structures (\ref{eqn:levels}).

\begin{Proposition}
If $\Gamma_{k} \rightrightarrows B_{k}$ is a weighted Lie groupoid of degree $k$, then we have the following tower of weighted groupoid structures of lower order and their morphisms:
\begin{gather*}
\begin{xy}
(0,20)*+{\Gamma_{k}}="a1"; (20,20)*+{\Gamma_{k-1}}="a2"; (40,20)*+{\cdots}="a3";   (60,20)*+{\Gamma_{1}}="a4";  (80,20)*+{\mathcal{G}}="a5"; %
(0,0)*+{B_{k}}="b1"; (20,0)*+{B_{k-1}}="b2";(40,0)*+{\cdots}="b3";  (60,0)*+{B_{1}}="b4"; (80,0)*+{M}="b5";%
{\ar "a1";"a2"}?*!/_3mm/{\tau^{k}};{\ar "a2";"a3"}?*!/_3mm/{\tau^{k-1}};{\ar "a3";"a4"}?*!/_3mm/{\tau^{2}};{\ar "a4";"a5"}?*!/_3mm/{\tau^{1}};
{\ar@<1.ex>"a1";"b1"} ;
{\ar@<-1.ex> "a1";"b1"} ?*!/^3mm/{\underline{s}_{k}} ?*!/_6mm/{\underline{t}_{k}};
{\ar@<1.ex>"a2";"b2"}; {\ar@<-1.ex> "a2";"b2"}?*!/^4mm/{\underline{s}_{k-1}} ?*!/_8mm/{\underline{t}_{k-1}}; {\ar@<1.ex>"a4";"b4"}; {\ar@<-1.ex> "a4";"b4"}?*!/^3mm/{\underline{s}_{1}} ?*!/_6mm/{\underline{t}_{1}}; {\ar@<1.ex>"a5";"b5"}; {\ar@<-1.ex> "a5";"b5"}?*!/^3mm/{\underline{\sigma}} ?*!/_6mm/{\underline{\tau}}; %
{\ar "b1";"b2"}?*!/^3mm/{\pi^{k}};{\ar "b2";"b3"}?*!/^3mm/{\pi^{k-1}};{\ar "b3";"b4"}?*!/^3mm/{\pi^{2}};{\ar "b4";"b5"}?*!/^3mm/{\pi^{1}};
\end{xy}
\end{gather*}
In particular, $\Gamma_{1} \rightrightarrows B_{1}$ is a $\mathcal{VB}$-groupoid.
\end{Proposition}

\subsection{Examples of weighted Lie groupoids}

In this subsection we sketch some simple examples of weighted Lie groupoids following adaptation of some of the standard examples of Lie groupoids and graded bundles. We expect more intricate examples to follow, some of which maybe more useful in applications.

\begin{Example}\label{exa:1}
All graded bundles can be considered as weighted Lie groupoids over themselves by taking all the structure maps to be the identity. This leads to the notion of a \emph{weighted unit Lie groupoid}.
\end{Example}

\begin{Example}\label{exa:2}
Just as Lie groups are examples of Lie groupoids (over a single point), we can consider \emph{weighted Lie groups} of degree $k$ which are Lie groups $G_{k}$ equipped with a compatible homogeneity structure of degree~$k$. The compatibility condition is simply $\rmh_{t}(g \cdot h) = \rmh_{t}(g) \cdot\rmh_{t}(h) $  for all elements $g, h \in G_{k}$.
\end{Example}

\begin{Example}\label{exa:3}
If $F_{k} \rightarrow M$ is a graded bundle of degree~$k$, then we can construct the \emph{weighted pair groupoid} in the obvious way; $F_{k}\times F_{k} \rightrightarrows F_{k}$. Underlying this is the `genuine' pair groupoid $M \times M \rightrightarrows M$. Moreover, as we have a series of af\/f\/ine f\/ibrations associated with any graded bundle the weighted pair groupoid can be `f\/ibred' as $F_{k}\times_{F_{q}} F_{k} \rightrightarrows F_{k}$ in the obvious way.
\end{Example}

\begin{Example}\label{exa:4}
Let $\mathcal{G} \rightrightarrows M$ be a Lie groupoid. Then associated with this is the \emph{higher tangent Lie groupoid}  $\sT^{k}\mathcal{G} \rightrightarrows \sT^{k}M$, which is naturally a weighted Lie groupoid and is formed by applying the higher tangent functor of order $k$ to the structure maps of $\mathcal{G} \rightrightarrows M$. Clearly the initial Lie groupoid is the `genuine' Lie groupoid underlying the higher tangent Lie groupoid.
\end{Example}

\begin{Example}\label{exa:5}
Similar to the previous example, one can consider iterating  the tangent functor $k$-times to construct the \emph{iterated tangent Lie groupoid} $\sT^{(k)}\mathcal{G}\rightrightarrows \sT^{(k)}M$. By passing to total weight we have a weighted Lie groupoid of degree $k$.
\end{Example}

\begin{Example}\label{exa:6}
As already commented on,  $\mathcal{VB}$-groupoids are weighted Lie groupoids of degree~$1$.  Standard examples of $\mathcal{VB}$-groupoids include $\sT \mathcal{G} \rightrightarrows \sT M$ and $\sT^{*}\mathcal{G} \rightrightarrows \textnormal{A}^{*}(\mathcal{G})$ for any Lie groupoid $\mathcal{G} \rightrightarrows M$.
\end{Example}

\begin{Example}\label{exa:7}
Combining Examples~\ref{exa:4} and~\ref{exa:6}, we can consider the \emph{higher Pontryagin Lie groupoid}
\begin{gather*}
\sT^{k}\mathcal{G} \times_{\mathcal{G}} \sT^{*}\mathcal{G} \rightrightarrows \sT^{k}M \times_{M} \rmA^{*}(\mathcal{G}),
\end{gather*}
 as a weighted Lie groupoid by passing from the natural bi-weight to total weight.
\end{Example}

\begin{Example}\label{exa:8}
Let $(G_{l}, \rmh)$ be a weighted Lie group and $(F_{k}, \rmg)$ be a graded bundle. We assume that $G_{l}$ acts on $F_{k}$. Then we can construct the \emph{weighted action Lie groupoid} as follows; $G_{l}\times F_{k} \rightrightarrows F_{k}$, where the source and target maps are given by
\begin{gather*}
\underline{s}(g,x)= x, \qquad
 \underline{t}(g,x) = gx,
\end{gather*}
  and the (partial) multiplication being
\begin{gather*}
(h, y)\cdot (g, x) = (hg, x).
\end{gather*}
We def\/ine $\underline{\rmh} := (\rmh, \rmg)$ as the homogeneity structure on the Lie groupoid. Compatibility with the source map and partial multiplication is automatic from the def\/inition of a homogeneity structure, but compatibility with the target map requires the natural condition $ \rmg_{t}(g x) =(\rmh_{t}g)(\rmg_{t}x)$. Thus we will def\/ine a weighted action Lie groupoid via the construction of an action Lie groupoid, with the addition of compatible homogeneity structures as described above.
\end{Example}

\begin{Example}\label{exa:9}
Any $\mathcal{GL}$-bundle $(D_{k}, \widehat{\rmh}, \widehat{\rml})$ can be considered as a weighed Lie groupoid in the following way. The base is def\/ined as $B_{k-1} := \widehat{\rml}_{0}(D_{k})$ and the source and target maps are set to the projection $\underline{s} = \underline{t} = \widehat{\rml}_{0}$. Thus, the $\underline{s}$-f\/ibres have the structure of graded spaces, cf.~\cite{Grabowski2012} (i.e., graded bundles over a point). The partial multiplication is just the addition in the $\underline{s}$-f\/ibres.  The compatibility of the Lie groupoid and graded bundle structures follows directly from the def\/inition of a $\mathcal{GL}$-bundle.
\end{Example}

\begin{Example}\label{exa:10}
The f\/ibred product
\begin{gather*}
\sT \mathcal{G} \times_{\mathcal{G}} \sT \mathcal{G} \times_{\mathcal{G}} \cdots \times_{\mathcal{G}} \sT \mathcal{G} \rightrightarrows \sT M \times_{M} \sT M \times_{M} \cdots \times_{M}\sT M,
\end{gather*}
where $\mathcal{G}\rightrightarrows M$ is an arbitrary Lie groupoid, can be considered as a weighted Lie groupoid in the following way. The natural homogeneity structure associated with the $i$-th factor $\sT \mathcal{G} \rightrightarrows \sT M$ can be shifted by $i \in \mathbb{N}$ and then, by passing to total weight, we have a weighted Lie groupoid of degree~$k$, assuming we the f\/ibred product consists of~$k$ factors.
\end{Example}

\section{Lie theory for weighted groupoids and algebroids}\label{sec:Lie theory}

We now turn our attention to the Lie theory relating weighted Lie groupoids and weighted Lie algebroids. The dif\/ferentiation of a weighted Lie groupoid produces rather naturally a weighted Lie algebroid. One can use the standard construction here, just taking a little care with the weights. The integration of a weighted Lie algebroid to a weighted Lie groupoid also follows via the standard constructions, again taking care with the graded structure. We will not address the full question of global integrability, rather we will assume integrability as a Lie algebroid and show that the weighted structure is naturally inherited.

\subsection{Dif\/ferentiation of weighted Lie groupoids}
It is clear  that as a Lie groupoid a weighted Lie groupoid always admits a Lie algebroid associated with it following the classical constructions. The question is \emph{if the associated Lie algebroid is in fact a weighted Lie algebroid}?

\begin{Theorem}\label{thm:groupoids to algebroids}
If $\Gamma_{k} \rightrightarrows B_{k}$ is a weighted Lie groupoid of degree $k$ with respect to a homogeneity structure $h$ on $\Gamma_k$, then $\rmA(\Gamma_{k}) \rightarrow B_{k}$ is a weighted Lie algebroid of degree $k+1$ with respect to the homogeneity structure $\widehat{h}$ defined by
\begin{gather}\label{wh}
\widehat{h}_t=(h_t)'=\Lie(h_t) .
\end{gather}
\end{Theorem}

\begin{proof}
We need only follow the classical constructions taking care of the weight as we proceed. In particular, we need only show that $\rmA(\Gamma_{k})$ has the structure of a $\mathcal{GL}$-bundle and that the corresponding Lie algebroid brackets are of weight~$-(k+1)$.

 1.~As  $\rmA(\Gamma_{k}) :=  \ker ( \sT\underline{s}  )  |_{\iota(B_{k})}$, it is clear that as a subbundle of $\sT \Gamma_{k}$ that we have the structure of a  $\mathcal{GL}$-bundle. The degree  $k$ homogeneity structure $\widehat{h}$ is inherited from the tangent lift of the homogeneity structure on $\Gamma_{k}$ and the degree $1$ homogeneity structure is inherited from the natural one related to the vector bundle structure of tangent bundles. In other words, $\widehat{h}_t=(h_t)'$ is the `dif\/ferential' of~$h$.

 2.~Sections of weight $r$  are by our def\/inition functions on $\rmA(\Gamma_{k})^{*}$ of bi-weight $(r-1,1)$, see Section~\ref{sec:Preliminaries}.  Thus, associating sections with right invariant vector f\/ields requires a~shift in the weights of $(-k, -2)$, which comes from associating momenta with derivatives, i.e., use weighted principle symbols. Thus, the prolongation of a~section of weight~$r$ to a~left invariant vector f\/ield is a~vector f\/ield  of bi-weight $(r-1-k,-1)$.  Thus if we take the Lie bracket of two sections of weight~$r_{1}$ and $r_{2}$ as right invariant vector f\/ields the resulting will be a  right invariant vector f\/ield of weight $(r_{1}+ r_{2}- 2 - 2k, -1)$. Now upon restriction and  `shifting back' to functions on~$\rmA(\Gamma_{k})^{*}$, the resulting section is of bi-weight $(r_{1}+ r_{2} -1 -(k+1), 1)$. Thus the Lie algebroid bracket carries weight $-(k+1)$.
\end{proof}

Theorem~\ref{thm:groupoids to algebroids} tells us that weighted groupoids are the objects that integrate weighted algebroids, but of course it does not tell us that we can actually (globally) integrate weighted Lie algebroids.  As a `corollary' we see that $\mathcal{VB}$-groupoids are the objects that integrate $\mathcal{VB}$-algebroids.

\begin{Example}
The weighted Lie algebroid associated with the weighted pair Lie groupoid $F_{k}\times F_{k} \rightrightarrows F_{k}$  is the tangent bundle $\sT F_{k}$.
\end{Example}

\begin{Example}
The $k$-th order tangent bundle of a Lie groupoid $\sT^{k}\mathcal{G}$ naturally comes with the structure of a weighted Lie groupoid over $\sT^{k}M$. The associated Lie algebroid is $\textnormal{A}(\sT^{k}\mathcal{G}) \simeq \sT^{k}\textnormal{A}(\mathcal{G})$, see  for example~\cite{KouotchopWamba:2013} for details. From~\cite{Bruce:2014} we know that that  $\sT^{k}\textnormal{A}(\mathcal{G})$ is a weighted Lie algebroid.
\end{Example}

\begin{Proposition}
Via the appropriate forgetful functors, the Lie functor restricts to the sub\-ca\-te\-gories of weighted Lie groupoids and weighted Lie algebroids, i.e.,
\begin{gather*}
\catname{wGrpd} \stackrel{\Lie}{\xrightarrow{\hspace*{30pt}}} \catname{wAlgd}.
\end{gather*}
\end{Proposition}

\begin{proof}
This follows from Theorem~\ref{thm:groupoids to algebroids} and the fact that both weighted Lie groupoid and weighted Lie algebroid morphisms intertwine the respective actions of the homogeneity structure.
\end{proof}

\subsection{Integration  of weighted Lie algebroids}
 We will not address the full problem of integration of Lie algebroids here, one should consult Crainic and  Fernandes~\cite{Crainic:2003} for details of the obstruction to (global) integrability.  For a details of the obstruction to the integration of $\mathcal{VB}$-algebroids see~\cite{Brahic:2014}. We remark though, that passing to the world of dif\/ferentiable stacks and the notion of a Weinstein groupoid all Lie algebroids are integrable~\cite{Tseng:2006}, though in this work we will strictly stay in the world of Lie groupoids. The question we address is not one of the integrability of weighted Lie algebroids as Lie algebroids, but rather if they integrate to weighted Lie groupoids. That is, \emph{does the homogeneity structure on the weighted Lie algebroid integrate to a homogeneity structure on the associated source simply-connected Lie groupoid such that we have the structure of a weighted Lie groupoid?}

 \begin{Theorem}
Let $D_{k+1} \rightarrow B_{k}$ be a weighted Lie algebroid of degree $k+1$ with respect to a~homogeneity structure $\widehat{\rm h}$ and $\Gamma_{k}$ its source simply-connected integration groupoid.
Then $\Gamma_{k}$ is a~weighted Lie groupoid of degree~$k$ with respect to the homogeneity structure~${\rm h}$ uniquely determined by~\eqref{wh}.
\end{Theorem}

\begin{proof}
A weighted Lie algebroid  is a Lie algebroid together with a homogeneity structure   $\widehat{\rmh}_{t}\colon D_{k+1} \rightarrow D_{k+1}$ such that  for  any $t\neq 0$ the homogeneity structure acts as  a  Lie algebroid automorphism, and for $t=0$  the homogeneity structure is a Lie algebroid morphism to the underlying non-graded Lie algebroid. Thus we can use Lie~II to uniquely integrate $\widehat{\rmh}_{t}$ (for a~given~$t$) to a Lie groupoid morphism
\begin{gather*}
\rmh_{t}\colon \  \Gamma \rightarrow \Gamma,
\end{gather*}
 such that ($\rmh_{t})' = \widehat{\rmh}_{t}$.
 The corresponding action $\rmh\colon {\mathbb R}\times\Gamma\to\Gamma$ is smooth. Indeed, the smooth action $\widehat{\rmh}$ on the Lie algebroid $D_{k+1}$ induces clearly a smooth action on the Banach manifold of $C^1$-admissible paths in $D_{k+1}$ whose projection on~$\Gamma$ gives~$\rmh$.
The uniqueness of the integration of Lie algebroid morphisms   together with the functorial properties of dif\/ferentiating Lie groupoids,  namely $(\rmh_{t} \circ \rmh_{s})'= \rmh'_{t} \circ \rmh'_{s}$, implies that  $\rmh\colon \mathbb{R} \times  \Gamma \rightarrow \Gamma$ is a homogeneity structure. Thus $\Gamma \rightrightarrows B$ is a weighted Lie groupoid.

 The only remaining question is that of the degree of the homogeneity structure~$\rmh$. This in fact follows directly from the Lie functor. The degree of the homogeneity structure is a local question and so we can examine everything in local coordinates.  In particular, on $\Gamma \rightrightarrows B$ we can employ local  coordinates adapted to the source f\/ibration
\begin{gather*}
\big(x^{\alpha}_{u},  Y_{u}^{I}\big),
\end{gather*}
\noindent for $0 \leq u \leq k'$, for some $k'$ to be determined. As the homogeneity structure commutes with the source, for simplicity we can assume that the weight of both coordinates run over the same range.  Following the construction of the associated Lie algebroid, we equip $\sT \Gamma$ with homogeneous coordinates
\begin{gather*}
\big(\underbrace{x^{\alpha}_{u}}_{(u,0)},  \underbrace{Y_{u}^{I}}_{(u,0)},  \underbrace{\dot{x}^{\beta}_{u+1}}_{(u,1)},  \underbrace{\dot{Y}_{u+1}^{J}}_{(u,1)} \big),
 \end{gather*}
where we have applied the tangent lift of the homogeneity structure to def\/ine the weight of the coordinates on the tangent bundle. The second component is simply the natural weight induced by the vector bundle structure. Locally we obtain $D_{k+1} \rightarrow B_{k}$ by setting $\dot{x} =0$ and $Y =0$.

Note that as $D_{k+1}$ is a graded bundle of degree $k$ (forget the linear structure) and that $(x, \dot{Y})$ give homogeneous coordinates on the weighted Lie algebroid $D_{k+1} \subset \sT \Gamma$. The tangent bundle of a graded bundle of degree $k'$ is also a graded bundle of degree $k'$ (again forgetting the linear structure).  Meaning that the weight of the coordinates $\dot{Y}$ on the algebroid are directly inherited from the coordinates $Y$ on the groupoid.  Thus we conclude that $k' = k$.   Upon setting $\Gamma = \Gamma_{k}$ we have established the theorem.
\end{proof}

\begin{Remark}
The results of  Bursztyn  et al.~\cite{Bursztyn:2014} on integration of $\mathcal{VB}$-algebroids rely on some technical issues regarding regular homogeneity structures. In particular it is some work to show  that regularity is preserved under integration.  By considering the larger category of graded bundles and not just vector bundles we bypass this technicality. The fact that integration of a weighted Lie algebroid produces a weighted Lie groupoid is quite clear; the degree of the integrated homogeneity structure then follows from the graded bundle structure of tangent bundle of the integration Lie groupoid rather than any dif\/f\/icult technical arguments.  The philosophy is that graded bundles give us a `little more room to manoeuvre'  as compared with just vector bundles.
\end{Remark}

Because $\textnormal{A}(\rmh_{0}(\Gamma_{k})) = \widehat{\rmh}_{0}(D_{k+1})$, it is clear that the integrability of $D_{k+1}\rightarrow B_{k}$ as a Lie algebroid implies that $A \rightarrow M$, where $A = \widehat{\rm h}_{0}(D_{k+1})$ is also integrable as a Lie algebroid.  The converse is not true and examples of $\mathcal{VB}$-algebroids with integrable base Lie algebroids for which the total Lie algebroid is not integrable exist, see \cite{Brahic:2014,Bursztyn:2014}.

\section{Weighted Poisson--Lie groupoids and their Lie theory}\label{sec:Weighted Poisson Groupoids}
Following our general ethos that we can construct weighted versions of our favorite geometric structures by simply including a compatible homogeneity structure, we now turn to the notion of weighted Poisson--Lie groups and weighted Lie bi-algebroids.  As the propositions in this section follow by mild adaptation of standard proofs by including weights into considerations, we will not generally present details of the proofs.

\subsection{Weighted Poisson--Lie groupoids}

\begin{Definition}
A \emph{weighted Poisson--Lie groupoid} $(\Gamma_{k}, \Lambda)$ of degree $k$ is a Poisson--Lie groupoid equipped with a multiplicative homogeneity structure $\underline{\rmh}\colon  \mathbb{R}\times \Gamma_{k} \rightarrow \Gamma_{k}$ of degree~$k$, such that the Poisson structure is of weight~${-}k$.
\end{Definition}

Recall that a Poisson--Lie groupoid is a Lie groupoid equipped with a multiplicative Poisson structure; this means that
\begin{gather*}
\textnormal{graph}(\textnormal{mult}) = \{(g,h,  g\circ h) \,| \, \underline{s}(g) = \underline{t}(h)  \}
\end{gather*}
 is coisotropic inside $ ( \Gamma_{k} \times \Gamma_{k}, \times \Gamma_{k}, \Lambda\oplus \Lambda \oplus {-}\Lambda )$.  The associated Poisson bracket on $C^{\infty}(\Gamma_{k})$ is of weight ${-}k$.
Weighted Poisson--Lie groupoids should be thought of as the natural generalisation on $\mathcal{PVB}$-groupoids as f\/irst def\/ined by Mackenzie \cite{Mackenzie:1999}. A  $\mathcal{PVB}$-groupoid is a $\mathcal{VB}$-groupoid equipped with a compatible Poisson structure.

An alternative, yet equivalent and illuminating way to view weighted Poisson--Lie groupoids is as follows. Let $(G, \Lambda)$ be a Lie groupoid equipped with a Poisson structure, we will temporarily forget the graded structure and place no conditions on the Poisson structure.  Associated with the Poisson structure is the f\/ibre-wise map
\begin{gather*}
\Lambda^{\#}\colon \  \sT^{*}G \rightarrow \sT G.
\end{gather*}
 This f\/ibre-wise map is of course completely independent of the groupoid structure on $G$.
Note, however, that if  $G$ is a Lie groupoid, then both $\sT^{*}G \rightrightarrows \textnormal{A}^{*}(G)$ and $\sT G \rightrightarrows \sT M$ are naturally Lie groupoids~\cite{Mackenzie2005}.  Then one can then def\/ine a Poisson--Lie groupoid as a Lie groupoid equipped with a Poisson structure such that the induced map $\Lambda^{\#}$ is a morphism of Lie groupoids.

The only extra requirement  for the case of a weighted Poisson--Lie groupoid is that $ \Lambda^{\#}\colon \sT^{*}\Gamma_{k}$ $\rightarrow \sT \Gamma_{k}$ must be a morphisms of double graded bundles, and this forces the Poisson structure to be of weight~${-}k$. Furthermore note that $\sT^{*}\Gamma_{k}$ and $\sT \Gamma_{k}$ are both $\mathcal{GL}$-bundles, cf.\ Section~\ref{sec:Preliminaries}.

\begin{Example}
Rather trivially, any weighted Lie groupoid of degree $k$ is weighted Poisson--Lie groupoid of degree $k$ when equipped with the zero Poisson structure.
\end{Example}

\begin{Example}
Any Poisson--Lie groupoid can be considered as a weighted Poisson--Lie groupoid of degree $0$ by equipping it with the trivial homogeneity structure. Lie groups equipped with  multiplicative Poisson structures are specif\/ic examples.
\end{Example}

\begin{Example}
Consider a graded bundle of degree $k$ equipped with a Poisson structure of weight $-k$, $(F_{k}, \Lambda)$. The pair groupoid $F_{k}\times F_{k} \rightrightarrows F_{k}$ is a weighted Poisson--Lie groupoid of degree $k$ if we equip $F_{k}\times F_{k}$ with the `dif\/ference' Poisson structure.
\end{Example}

\begin{Example}
Consider a Poisson--Lie groupoid $(\mathcal{G} \rightrightarrows M, \pi)$, which we can consider as a~trivial weighted Poisson--Lie groupoid. Recall that $\sT^{k}\mathcal{G} \rightrightarrows \sT^{k}M$ is a weighted Lie groupoid of degree~$k$. Furthermore, in~\cite{KouotchopWamba:2013} it was shown that the higher tangent lift of a multiplicative Poisson structure is  a multiplicative Poisson structure itself. By inspection we see that Poisson structure lifted to the $k$-th order tangent bundle $ \Lambda := \rmd_{\sT}^{k}\pi$ is of weight $-k$, with respect to the natural homogeneity structure on the higher tangent bundle. Thus $(\sT^{k}\mathcal{G} \rightrightarrows \sT^{k}M, \rmd_{\sT}^{k}\pi )$ is a~weighted Poisson--Lie groupoid of degree~$k$. This example we consider  as the archetypal weighted Poisson--Lie groupoid.
\end{Example}

\subsection{Weighted Lie bi-algebroids}

Follow Roytenberg \cite{Roytenberg:2001}, Voronov \cite{Voronov:2001qf} and Kosmann-Schwarzbach  \cite{Kosmann-Schwarzbach:1995}  we are lead to the following def\/inition.

\begin{Definition}
A pair of weighted Lie algebroids $(D_{k}, D_{k}^{*})$  both of degree $k$  is a \emph{weighted Lie bi-algebroid} if and only if their exists odd Hamiltonians $\mathcal{Q}_{D_{k}}$ and $S_{D_{k}^{*}}$  on $\sT^{*}(\Pi D_{k})$ of weight $(k-1, 2,1)$ and $(k-1, 1,2)$, respectively,  that Poisson commute.
\end{Definition}

Let us employ homogeneous local coordinates
\begin{gather*}
\big( \underbrace{x_{u}^{\alpha}}_{(u,0,0)},  \underbrace{\theta_{u+1}^{I}}_{(u,1,0)},   \underbrace{p_{\beta}^{u+2}}_{(u,1,1)},   \underbrace{\chi_{J}^{u+1}}_{(u,0,1)} \big),
\end{gather*}
on $\sT^{*}(\Pi D_{k})$. In these homogeneous coordinates we have
\begin{gather*}
\mathcal{Q} = \theta_{u-u'+1}^{I} Q_{I}^{\alpha}[u'](x)p_{\alpha}^{k+1-u} + \frac{1}{2} \theta_{u''-u'+1}^{I} \theta_{u-u''+1}^{J}Q_{JI}^{K}[u'](x) \chi_{K}^{k-u},\\
S  = \chi_{I}^{u-u'+1} Q^{I \alpha}[u'](x) p_{\alpha}^{k+1-u} + \frac{1}{2}\chi_{I}^{u''-u'+1}\chi_{J}^{u-u''+1}Q_{K}^{JI}[u'](x)\theta_{k-u}^{K}.
\end{gather*}
Since we have a Schouten structure here of tri-weight $(k-1,1,2)$, we have a graded f\/ibre-wise map
\begin{gather*}
S^{\#} \colon \ \sT^{*}(\Pi D_{k}) \longrightarrow \Pi \sT(\Pi D_{k}),
\end{gather*}
given in local coordinates by
\begin{gather*}
(S^{\#})^{*}\rmd x^{\alpha}_{u+1}  =  \frac{\partial S}{\partial p_{\alpha}^{k+1-u}  }  ,\qquad
  (S^{\#})^{*} \rmd\theta_{u+2}^{I}  =   \frac{\partial S }{\partial \chi_{\alpha}^{k-u}},
\end{gather*}
where we have employed homogeneous local coordinates
\begin{gather*}
\big( \underbrace{x_{u}^{\alpha}}_{(u,0,0)},   \underbrace{\theta_{u+1}^{I}}_{(u,1,0)},   \underbrace{\rmd x^{\beta}_{u+1}}_{(u,0,1)},  \underbrace{\rmd\theta^{J}_{u+2}}_{(u,1,1)} \big),
\end{gather*}
on $\Pi\sT(\Pi D_{k})$.

\begin{Proposition}
Let $(D_{k}, D_{k}^{*})$ be a weighted Lie  bi-algebroid of degree~$k$. The fibre-wise map~$S^{\#}$ is a graded morphisms of $Q$-manifolds
\begin{gather*}
\big(\sT^{*}(\Pi D_{k}),   Q_{0} :=  \{\mathcal{Q} , \bullet \} \big) \longrightarrow \big( \Pi \sT(\Pi D_{k}),   \mathcal{L}_{Q} = [\rmd, i_{Q}] \big).
\end{gather*}
\end{Proposition}

\begin{proof}
It follows via direct computation and so we omit details.
\end{proof}

\begin{Remark} The above proposition holds for general $QS$-manifolds and the additional gradings play no critical role. In fact, we do not need a Schouten structure, but just an odd quadratic Hamiltonian such that $\mathcal{L}_{Q}S:= \{\mathcal{Q}, S \} =0$.
\end{Remark}

\begin{Example}
If $(\textnormal{A}, \textnormal{A}^{*})$ is a Lie bi-algebroid, then $(\sT^{k}\textnormal{A}, \sT^{k}\textnormal{A}^{*})$ is a Lie bi-algebroid of degree $k+1$. See for details~\cite{KouotchopWamba:2013}, although  the authors do not make mention of the graded structure.
\end{Example}

\begin{Example}
Consider a weighted Lie algebroid $D_{k}$ of degree $k$, together with an even function $\mathcal{P} \in C^{\infty}(\Pi D_{k}^{*})$ of bi-weight $(k-1, 2)$ such that $\SN{\mathcal{P}, \mathcal{P}} =0$, where the bracket is the Schouten bracket encoding the Lie algebroid structure. Such functions we consider as Poisson structures on weighted Lie algebroids. Then, following well-known classical results, $(D_{k}, D_{k}^{*})$ is a weighted Lie bi-algebroid of degree $k$. This example is the weighted version of a triangular Lie bi-algebroid. As a canonical example, $(\sT(\sT^{k-1}M), \sT^{*}(\sT^{k-1}M))$ is a weighted Lie bi-algebroid of degree $k$.  The homological vector f\/ield is supplied by the de Rham dif\/ferential on $\Pi \sT(\sT^{k-1}M)$ and the Schouten structure is that associated with the canonical Schouten bracket on $\Pi \sT^{*}(\sT^{k-1}M)$.
\end{Example}

\subsection[The Lie theory of weighted Poisson--Lie groupoids and weighted Lie bi-algebroids]{The Lie theory of weighted Poisson--Lie groupoids\\ and weighted Lie bi-algebroids}

It is well known that the inf\/initesimal object associated with a Poisson--Lie groupoid is a Lie bi-algebroid. Let us quickly sketch the association of a Lie bi-algebroid with a Poisson--Lie groupoid. The construction of the underlying Lie algebroid is as standard, we only need to show how the appropriate Schouten  structure is generated and that it carries the correct weight. With this in mind, let us outline the classical construction.

It can be shown that there are two natural isomorphisms
\begin{gather*}
 \Pi \vartheta \colon \  \Pi \textnormal{A}(\sT^{*}G)  \longrightarrow \sT^{*}(\Pi \textnormal{A}(G)), \qquad
 \Pi j \colon \ \Pi \sT (\Pi\textnormal{A}(G))  \longrightarrow  \Pi \textnormal{A}(\sT G),
\end{gather*}
 which follow from application of the parity reversion functor to the standard non-super isomorphisms (see \cite{Mackenzie2005, Mackenzie:2000} and \cite{KouotchopWamba:2013}, where the higher tangent versions are also discussed).  Let $(G, \Lambda)$ be a Poisson--Lie groupoid, then as we have  $\Lambda^{\#}\colon  \sT^{*}G \rightarrow \sT G$, we also have  the `superised' version:
\begin{gather*}
\Pi ( \Lambda^{\#})' \colon \ \Pi \textnormal{A}(\sT^{*}G) \rightarrow \Pi \textnormal{A}(\sT G).
\end{gather*}
Furthermore, the above gives  rise to a Schouten structure on $\textnormal{A}$ viz
\begin{gather*}
\Pi (\Lambda^{\#})' := \Pi j \circ S^{\#} \circ \Pi \vartheta,
\end{gather*}
where $S^{\#}\colon \sT^{*} (\Pi \textnormal{A}(G)) \longrightarrow \Pi \sT(\Pi \textnormal{A}(G))$. One now only has to check the weight when dealing with the weighted versions.

 In the other direction, it is also well known an integration procedure for passing from  Lie bi-algebroids to Poisson--Lie groupoids.  Using Lie's second theorem one can show that if $(A, A^{*})$ is a Lie bi-algebroid  and $G \rightrightarrows M$ is the source simply-connected Lie groupoid integrating $A \rightarrow M$, then there exists a unique Poisson structure on $G$ such that $(G,\Lambda)$ is a Poisson--Lie groupoid integrating $(A, A^{*})$. Again, it is straightforward to check the weights when dealing with the weighted versions. We are lead to the following proposition.
\begin{Proposition}
There is a canonical correspondence between weighted Poisson--Lie groupoids of degree~$k$ and
$($integrable$)$ weighted Lie bi-algebroids of degree~$k+1$.
\end{Proposition}

\subsection{The Courant algebroid associated with weighted Lie bi-algebroids}

Following our general ethos,  a weighted Courant algebroid is a Courant algebroid (see  \cite{Courant:1990,Liu:1997}) with an additional compatible homogeneity structure.
\begin{Definition}
A \emph{weighted Courant algebroid of degree $k$} consists of the following data:
 \begin{enumerate}\itemsep=0pt
\item A double graded supermanifold $({\mathcal M},\rmh,\rml)$ of bi-degree $(k-1,2)$ such that $({\mathcal M}, \rml)$ is an $N$-manifold (the Grassmann parity of homogeneous coordinates is determined by $\rml$).
\item An even  symplectic form  $\omega$ on $\mathcal{M}$  of bi-degree $(k-1,2)$.
\item An odd function $\Theta \in C^{\infty}(\mathcal{M})$ of bi-degree $(k-1,3)$, such that $\{\Theta, \Theta \}_\omega =0$, where the bracket is the Poisson bracket induced by~$\omega$.
\end{enumerate}
 The corresponding derived bracket
\begin{gather*}
\SN{\sigma^{1}, \sigma^{2}}_\Theta :=  \big\{\big\{ \sigma^1, \Theta \big\}_\omega, \sigma^{2} \big\}_\omega
\end{gather*}
is closed on functions of bi-degree $(k-1,1)$ (so functions of total degree $k$) and def\/ines the higher analog of the Courant--Dorfman bracket.
\end{Definition}

Our motivating example is associated with an arbitrary  weighted Lie bi-algebroid $(D_{k},D_{k}^{*})$. Let us `collapse' the tri-weight on $\sT^{*}(\Pi D_{k})$ to a bi-weight by taking the sum of the last two weight vector f\/ields. Then, the Hamiltonian vector f\/ield
\begin{gather*}
Q_{\lambda} :=  \{ \mathcal{Q} + \lambda   S, \bullet  \} \in \Vect(\sT^{*}(\Pi D_{k}))
\end{gather*}
 is of bi-weight $(0,1)$.  Here we take $\lambda \in \mathbb{R}$ to be some free parameter (carrying no weight) and so we have a pencil of such vector f\/ields. Clearly $Q_{\lambda}$ is a homological vector f\/ield.  Moreover, as~$Q_{\lambda}$ is a Hamiltonian vector f\/ield its Lie derivative annihilates the canonical symplectic structure on~$\sT^{*}(\Pi D_{k})$.   Note that the canonical symplectic form has natural bi-weight $(k-1,2)$. Thus, following Roytenberg~\cite{Roytenberg:2001}, we have  objects  that deserve to be called \emph{weighted Courant algebroids} of degree $k$.

To uncover the natural pairing, bracket and anchor structure, f\/irst note that $\Pi D_{k} \times_{B_{k-1}} \Pi D_{k}^{*}$ is a  quotient supermanifold of $\sT^{*}\Pi D_{k}$ def\/ined (locally) by projecting out the coordinates $p^{u+1}_{\alpha}$. Thus we will naturally identify functions on $\sT^{*}(\Pi D_{k})$ of bi-weight $(r-1,1)$ as homogeneous sections of  the vector bundle $D_k\oplus D^*_k=D_{k}\oplus_{B_{k-1}} D_{k}^{*}$ of degree $r$, up to a shift in the Grassmann parity.  In the homogeneous coordinates introduced above, any homogeneous section of degree $r$ is of the form
\begin{gather*}
\sigma_{r} = s_{r-u-1}^{I}(x)\chi_{I}^{u+1} + \theta_{u+1}^{I}s_{I}^{r-u-1}(x).
\end{gather*}

The \emph{natural pairing} between sections, that is a f\/ibre-wise pseudo-Riemannian structure on $D_{k}\oplus D_{k}^{*}$,  comes from the Poisson bracket on $\sT^{*}(\Pi D_{k})$,
\begin{gather*}
\langle\sigma^{1}, \sigma^{2} \rangle = \big\{ \sigma^{1}, \sigma^{2} \big\} \in C^{\infty}(B_{k-1}),
\end{gather*}
 and thus the pairing carries bi-weight $(-k+1, -2)$, or we can think of the  total weight as $-k-1$. The fact that we have identif\/ied homogeneous sections with particular odd functions on  $\sT^{*}(\Pi D_{k})$ does not ef\/fect the identif\/ication of the metric structure and the Poisson structure; everything is linear and so the parity reversion does not add any real complications here.
Let us simplify notation slightly and def\/ine  $\Theta_{\lambda} :=  \mathcal{Q} + \lambda   S$.  Then the  derived bracket, given by
\begin{gather*}
\SN{\sigma^{1}, \sigma^{2}}_{\lambda} :=  \big\{\big\{ \sigma^{1}, \Theta_{\lambda} \big\}, \sigma^{2} \big\},
\end{gather*}
carries bi-weight $(-k+1, -1)$  (so the total weight $-k$), similarly to the case of a weighted Lie algebroid. It is closed on functions of bi-degree $(k-1,1)$, so sections of the vector bundle $D_k\oplus D^*_k$ of degree~$k$. The latter is the `higher' analog of the \emph{Courant--Dorfman bracket}, which is not a Lie bracket, but rather a Loday (Leibniz) bracket. The anchor is then def\/ined as
\begin{gather*}
\rho_\lambda\colon \  \Sec( D_{k}\oplus D_{k}^{*} )  \rightarrow  \Vect(B_{k-1}),\qquad
  \rho_\lambda(\sigma)[f]  :=  \{\{ \sigma, \Theta_{\lambda} \}, f \},
\end{gather*}
 for all $f \in C^{\infty}(B_{k-1})$.

\begin{Theorem}
Via the above construction, if $(D_{k},D_{k}^{*})$ is a weighted Lie bi-algebroid of deg\-ree~$k$, then $\sT^{*}(\Pi D_{k})$ naturally has the structure of a weighted Courant algebroid of degree~$k$ $($setting $\lambda =1$ for simplicity$)$.
\end{Theorem}

\begin{Example}
The $k=1$ case is just that of the  Courant structure associated with a standard Lie bi-algebroid. The case of $k=2$ gives an example of a   $\mathcal{VB}$-Courant algebroid.
\end{Example}

\begin{Example}
As any weighted Lie algebroid $D_{k}$ can be considered as a weighted Lie bi-algebroid with the trivial structure on the dual bundle, $\sT^{*}(\Pi D_{k})$ can be considered as a weighted Courant algebroid with the obvious homological vector f\/ield.
\end{Example}

\begin{Example}
If we take $D_{k} = \sT(\sT^{k-1}M)$ then we have the natural structure of a weighted Courant algebroid given in local coordinates as
\begin{gather*}
Q =  \big\{ \theta_{u+1}^{\alpha} p_{\alpha}^{k+1-u}, \bullet \big\}= \theta_{u+1}^{\alpha}\frac{\partial}{\partial x^{\alpha}_{u} }  +   p_{\alpha}^{u+1}\frac{\partial }{\partial \chi_{\alpha}^{u+1} }  \in \Vect\big(\sT^{*}\big(\Pi \sT \big(\sT^{k-1} M\big)\big)\big).
\end{gather*}
  This example should be considered as an natural generalisation of the canonical Courant algebroid on the generalised tangent bundle $\sT M \oplus \sT^{*}M$, which is a substructure of $\sT(\sT^{*}M)$.
\end{Example}

\subsection{Remarks on contact and Jacobi groupoids}

The notion of a \emph{weighted symplectic groupoid} is clear: it is just a weighted Poisson groupoid with an invertible Poisson, thus symplectic, structure.
By replacing  the homogeneity structure, i.e., an action of the monoid ${\mathbb R}$ of multiplicative reals, with a smooth action of its subgroup $\mathbb{R}^{\times}={\mathbb R}{\setminus}\{ 0\}$, one obtains a principal $\mathbb{R}^{\times}$-bundle in the category of symplectic (in general Poisson)  groupoids. We get in this way the `proper', in our belief, def\/inition of a \emph{contact} (resp.\ \emph{Jacobi}) \emph{groupoid}.

This belief is based on the general and well-known \emph{credo}, presented, e.g., in~\cite{Grabowski2013}, that the geometry of contact (more generally, Jacobi) structures on a manifold $M$ is nothing but the geometry of `homogeneous symplectic' (resp.\ `homogeneous Poisson') structures on an ${\mathbb R}^\times$-principal bundle $P\to M$. The word `homogeneous' refers to the fact that the symplectic form (resp., Poisson tensor) is homogeneous with respect to the ${\mathbb R}^\times$-action.
Of course, it is pretty well known that any contact (resp.\ Jacobi) structure admits a symplectization (resp.\ poissonization) and these facts are frequently used in proving  theorems on contact and Jacobi structures (see, e.g.,~\cite{Mehta:2013}). However, in~\cite{Grabowski2013} the symplectization/poissonization was taken seriously as a genuine def\/inition of a contact (Jacobi) structure. The general \emph{credo} then says that any object related to a contact (Jacobi) structure ought to be considered  as the corresponding object in the symplectic (Poisson) category, equipped additionally with an ${\mathbb R}^\times$-principal bundle structure, compatible with other structures def\/ining the object. The only problem is what the compatibility means in various cases.

For Lie groupoids (and also Lie algebroids) is natural to expect that the condition of an ${\mathbb R}^\times$-action to be compatible with the groupoid (algebroid) structure is the same as for the action of multiplicative~${\mathbb R}$, i.e., it consists of groupoid (algebroid) isomorphisms. This is the reason why we comment on contact and Jacobi groupoids here. We present the  detailed study of contact/Jacobi groupoids in a~separate paper~\cite{Bruce:2015}.

At this point, note only  that our def\/inition of a contact groupoid turns out to be equivalent to the def\/inition of Dazord \cite{Dazord:1995}. The f\/irst and frequently used def\/inition, presented in~\cite{Kerbrat:1993}, is less general and involves an arbitrary multiplicative function, that is due to the fact that in this approach contact structures are forced to be trivial, i.e., to admit a global contact 1-form. Note that very similar idea of using `homogeneous symplectic' manifolds in the context of contact groupoids appears already in~\cite{Crainic:2007} in a slightly less general framework of ${\mathbb R}_+$-actions.

To f\/inish, we want to stress that our use of the group ${\mathbb R}^\times$ rather than ${\mathbb R}_+$ is motivated by the need of including non-trivial line bundles into the picture; there is a nice one-to-one correspondence between principal ${\mathbb R}^\times$-bundles and line bundles. In this understanding, the contact structure does not need to admit a global contact form and the Jacobi bracket is def\/ined on sections of a line bundle rather than on functions, so it is actually a \emph{local Lie algebra} in the sense of Kirillov~\cite{Kirillov:1976} or a \emph{Jacobi bundle} in the sense of Marle~\cite{Marle:1991}. This, in turn, comes from the observation (cf.~\cite[Remark~2.4]{Grabowski2013}) that the Jacobi bracket is \emph{par excellence} a Lie bracket related to a module structure, even if the regular module action of an associative commutative algebra on itself looks formally as the multiplication in the algebra. Consequently, allowing nontrivial modules (line bundles) is necessary for the full and correct picture.

\subsection*{Acknowledgements}

The authors are indebted to the anonymous referees whose comments have served to improve the content and presentation of this paper. The research of K.~Grabowska and J.~Grabowski  was funded by the  Polish National Science Centre grant under the contract number DEC-2012/06/A/ST1/00256. A.J.~Bruce graciously acknowledges the f\/inancial support of the Warsaw Centre for Mathematics and Computer Science in the form of a postdoctoral fellowship.

\pdfbookmark[1]{References}{ref}
\LastPageEnding

\end{document}